\documentclass[11pt]{article}
\usepackage{amsfonts,amssymb,amsmath,amsthm,bm,mathtools}
\usepackage{latexsym,url,color,enumerate}
\usepackage{array,amstext}
\usepackage{calc}
\usepackage{wrapfig}
\usepackage{graphicx}
\graphicspath{{figs/}{../figs/}}
\usepackage{mwe}
\newtheorem{theorem}{Theorem}[section]
\newtheorem{proposition}[theorem]{Proposition}
\newtheorem{lemma}[theorem]{Lemma}
\newtheorem{warning}[theorem]{Warning}
\newtheorem{corollary}[theorem]{Corollary}
\newtheorem{notation}[theorem]{Notation}

\theoremstyle{definition}
\newtheorem{definition}[theorem]{Definition}

\theoremstyle{remark}
\newtheorem{remark}[theorem]{Remark}
\newcommand{\RNum}[1]{\uppercase\expandafter{\romannumeral #1\relax}}

\newcommand{\ZZ}{\mathbb{Z}}
\newcommand{\NN}{\mathbb{N}}
\newcommand{\QQ}{\mathbb{Q}}
\newcommand{\RR}{\mathbb{R}}
\newcommand{\FF}{\mathbb{F}}
\newcommand{\diag}{\textup{diag}}

\newcommand{\Min}{\textup{Min}}
\newcommand{\BW}{\textup{BW}}
\newcommand{\SBW}{\textup{sBW}}
\newcommand{\Dcyc}{{\textup{D}}^{\textup{(cyc)}}}
\newcommand{\GL}{\textup{GL}}
\newcommand{\Aut}{\textup{Aut}}
\newcommand{\supp}{\textup{supp}}
\newcommand{\wt}{\textup{wt}}
\newcommand{\Aff}{\textup{Aff}}
\newcommand{\Spin}{\textup{Spin}}
\newcommand{\GU}{\Gamma \textup{U}}
\newcommand{\SO}{\textup{SO}}
\newcommand{\so}{\textup{so}}
\newcommand{\dist}{\textup{dist\,}}

\newcommand\nc\newcommand
\nc\bfa{{\boldsymbol a}}\nc\bfA{{\bf A}}\nc\cA{{\mathcal A}}
\nc\bfb{{\boldsymbol b}}\nc\bfB{{\bf B}}\nc\cB{{\mathcal B}}
\nc\bfc{{\boldsymbol c}}\nc\bfC{{\bf C}}\nc\cC{{\mathcal C}}
\nc\bfd{{\boldsymbol d}}\nc\bfD{{\bf D}}\nc\cD{{\mathcal D}}
\nc\bfe{{\boldsymbol e}}\nc\bfE{{\bf E}}\nc\cE{{\mathcal E}}
\nc\bff{{\boldsymbol f}}\nc\bfF{{\bf F}}\nc\cF{{\mathcal F}}
\nc\bfg{{\boldsymbol g}}\nc\bfG{{\bf G}}\nc\cG{{\mathcal G}}
\nc\bfh{{\boldsymbol h}}\nc\bfH{{\bf H}}\nc\cH{{\mathcal H}}
\nc\bfi{{\boldsymbol i}}\nc\bfI{{\bf I}}\nc\cI{{\mathcal I}}
\nc\bfj{{\boldsymbol j}}\nc\bfJ{{\bf J}}\nc\cJ{{\mathcal J}}
\nc\bfk{{\boldsymbol k}}\nc\bfK{{\bf K}}\nc\cK{{\mathcal K}}
\nc\bfl{{\boldsymbol l}}\nc\bfL{{\bf L}}\nc\cL{{\mathcal L}}
\nc\bfm{{\boldsymbol m}}\nc\bfM{{\bf M}}\nc\cM{{\mathcal M}}
\nc\bfn{{\boldsymbol n}}\nc\bfN{{\bf N}}\nc\cN{{\mathcal N}}
\nc\bfo{{\boldsymbol o}}\nc\bfO{{\bf O}}\nc\cO{{\mathcal O}}
\nc\bfp{{\boldsymbol p}}\nc\bfP{{\bf P}}\nc\cP{{\mathcal P}}
\nc\bfq{{\boldsymbol q}}\nc\bfQ{{\bf Q}}\nc\cQ{{\mathcal Q}}
\nc\bfr{{\boldsymbol r}}\nc\bfR{{\bf R}}\nc\cR{{\mathcal R}}
\nc\bfs{{\boldsymbol s}}\nc\bfS{{\bf S}}\nc\cS{{\mathcal S}}
\nc\bft{{\boldsymbol t}}\nc\bfT{{\bf T}}\nc\cT{{\mathcal T}}
\nc\bfu{{\boldsymbol u}}\nc\bfU{{\bf U}}\nc\cU{{\mathcal U}}
\nc\bfv{{\boldsymbol v}}\nc\bfV{{\bf V}}\nc\cV{{\mathcal V}}
\nc\bfw{{\boldsymbol w}}\nc\bfW{{\bf W}}\nc\cW{{\mathcal W}}\nc\sW{{\mathscr W}}
\nc\bfx{{\boldsymbol x}}\nc\bfX{{\boldsymbol X}}\nc\cX{{\mathcal X}}\nc\sX{{\mathscr X}}
\nc\bfy{{\boldsymbol y}}\nc\bfY{{\boldsymbol Y}}\nc\cY{{\mathcal Y}}\nc\sY{{\mathscr Y}}
\nc\bfz{{\boldsymbol z}}\nc\bfZ{{\boldsymbol Z}}\nc\cZ{{\mathcal Z}}\nc\sZ{{\mathscr Z}}

\title{Strongly perfect lattices sandwiched between Barnes-Wall lattices}
\author{Sihuang Hu \thanks{husihuang@gmail.com, Humboldt fellow supported by the AvH foundation.}, Gabriele Nebe
\thanks{nebe@math.rwth-aachen.de} }

\begin{document}

\maketitle

{\small 
{\sc Abstract.} 
New series of $2^{2m}$-dimensional universally strongly perfect lattices 
$\Lambda_I $ and $\Gamma_J $ are constructed
with
$$2\BW_{2m} ^{\#}  \subseteq \Gamma _J \subseteq \BW_{2m} \subseteq \Lambda _I \subseteq \BW _{2m}^{\#}  .$$
The lattices are found by 
restricting the spin representations of the automorphism group
of the Barnes-Wall lattice to its subgroup 
${\mathcal U}_m:={\mathcal C}_m (4^H_{\bf 1}) $. 
The group ${\mathcal U}_m$ is 
the Clifford-Weil group associated to the Hermitian self-dual 
codes over $\FF _4$ containing ${\bf 1}$, so the ring of
polynomial invariants of ${\mathcal U}_m$ is spanned by 
the genus-$m$ complete weight enumerators of such codes. 
This allows us to show that all the ${\mathcal U}_m$ invariant
lattices are universally strongly perfect.
We introduce a new construction, $\Dcyc $, 
for chains of (extended) cyclic codes to
obtain (bounds on) the minimum of the new lattices. 
}

\section{Introduction}

The famous Barnes-Wall lattices $\BW _{2m}$ of dimension $2^{2m}$ (with $m\in \NN$) form an 
important infinite family of even lattices. 
They have several constructions allowing to determine 
discriminant group and minimum
 $$\BW_{2m}^{\#} / \BW _{2m} \cong \FF_2 ^{2^{2m-1}} ,  \
\min (\BW_{2m}) = 2^{m} ,$$  and even the kissing number and the
shortest vectors in a very explicit way \cite{BarnesWall}, \cite{BroueEnguehard}. 
Also their automorphism groups 
$${\mathcal G}_{2m} := \Aut (\BW_{2m}) \cong 2^{1+4m}_+ . O_{4m}^+(2) $$ 
are of relevance in  various places: 
\\
The groups ${\mathcal G}_{2m} $ are maximal finite subgroups of 
$\GL _{2^{2m}} (\QQ )$ all of whose invariant lattices are 
scalar multiples of $\BW _{2m}$ and its dual $\BW_{2m}^{\#} $. 
The lattice $\BW_{2m}$ is 2-modular
in the sense of \cite{Quebbemann}, 
i.e. there is a  similarity $h$ of norm $1/2$ with 
$h(\BW_{2m} ) = \BW_{2m}^{\#} $. Then $h$ is in the 
normalize of ${\mathcal G}_{2m}$ in $\GL _{2^{2m}} (\QQ )$
(see \cite{norm}). 
The group ${\mathcal C}_{2m} := 
{\mathcal G}_{2m} . \langle \sqrt{2} h \rangle $
is the real Clifford group  (see \cite{CliffordNRS}) 
whose ring of invariant polynomials is spanned by the 
genus $2m$ complete weight enumerators of self-dual binary codes. 
This identification is used in \cite{Bachoc} to deduce that 
all layers of the Barnes-Wall lattices form spherical 6-designs,
showing that the Barnes-Wall lattices are universally strongly perfect lattices.
In particular $\BW_{2m}$ realizes a local maximum of the density function
on the space of all similarity classes of $2^{2m}$-dimensional lattices
(see \cite{Venkov}). 
In the present paper we construct new infinite series of lattices 
$\Lambda_I $ and $\Gamma_J $ with 
$$2\BW_{2m} ^{\#}  \subseteq \Gamma _J \subseteq \BW_{2m} \subseteq \Lambda _I \subseteq \BW _{2m}^{\#}  $$
for subsets $I,J\subseteq \{ 0,\ldots  , m \}$ such that 
 $m-i$ is odd and $m-j$ is even for all $i\in I$, $j\in J$.
We call them \emph{sandwiched} lattices, as 
they are sandwiched between two Barnes-Wall lattices. 
For $m\geq 3$ the densest of these lattices 
is $\Lambda _{I_0}$  for 
$I_0 :=  \{m-i \mid m\geq i\geq 3, i \mbox{ odd }  \} $, 
whose minimum is the same as $\min (\BW _{2m})$; in particular 
these lattices are denser than the Barnes-Wall lattices.

To find these lattices we consider the sandwiched lattices that 
are invariant under the subgroup 
$$ {\mathcal C}_{m}(4^{H}_{\bf 1}) =
2^{1+4m}_+ . \GU _{2m} (\FF _4) =: {\mathcal U}_{m} \leq {\mathcal G}_{2m} .$$
The group ${\mathcal U}_m$ is the genus-$m$ Clifford-Weil group 
${\mathcal C}_{m}(4^{H}_{\bf 1}) $
associated to the Type of Hermitian self-dual codes over $\FF_ 4$ that 
contain the all ones vector. 
As in \cite{Bachoc} the invariant theory of this Clifford-Weil group
allows to predict that all its invariant lattices are universally 
strongly perfect (see Section \ref{usp} for more details). 
To obtain some information about these lattices, we restrict the 
spin representations $\BW_{2m} ^{\#} / \BW_{2m }$ respectively 
$\BW_{2m} / 2\BW_{2m}^{\#} $ of the orthogonal group $O_{4m}^+(\FF _2)$ 
to its subgroup $\GU_{2m} (\FF _4)$. 
It turns out that these restrictions are both  multiplicity free and all their 
composition factors are absolutely irreducible self-dual modules, 
$Y_k$ ($k\in \{0,\ldots , m\}$, $m-k$ odd respectively even).
Theorem \ref{latuni} gives a parametrization of the 
${\mathcal U}_m$ invariant sandwiched lattices.  
In particular for $m=2$ we discover a new pair of universally strongly 
perfect lattices $\Gamma _{\{ 2 \}}$ and 
$2\Gamma _{\{2 \} }^{\#} = \Gamma _{\{ 0 \}} $ in dimension 16
thus adding the first new  entry to \cite[Tableau 19.1]{Venkov} which 
was created 20 years ago. 

One way to construct $\BW_{2m}$ is by applying Construction D
to a  suitable basis of a 
chain of Reed-Muller codes. The Reed-Muller codes are extended cyclic
codes for which the minimum distance is obtained by the well known BCH bound. 
This cyclic permutation, say  $\sigma $, plays the key role 
in constructing and identifying 
the ${\mathcal U}_m$ invariant sandwiched lattices. 
It defines an automorphism of all Reed-Muller codes of the given length 
and also of the Barnes-Wall lattices, more precisely 
$$\sigma \in {\mathcal U}_m \subseteq {\mathcal G}_{2m} .$$
The eigenvalues of $\sigma $ on the simple ${\mathcal U}_m$ modules $Y_k$ 
indicate which chains of extended cyclic overcodes 
of the Reed-Muller codes we need to take to  
obtain the ${\mathcal U}_m$ invariant sandwiched lattices from Theorem \ref{latuni}.

The main problem of Construction D is that it depends not only
on the chain of codes but also on the choice of suitable bases. 
For chains of (extended) cyclic codes over prime fields, however, 
there is a unique way, which we call Construction $\Dcyc $, to 
define a lattice that is again invariant under the cyclic permutation
(see Section \ref{constAD}).
This construction also yields (lower bounds on) the minimum of 
the lattices $\Gamma _J$ and $\Lambda _I$ (Theorems \ref{sandlatdetails} and
\ref{GammaequalsL}).

\section{Preliminaries} 

\subsection{Cyclic codes}\label{Sec:cyc} 

Let $q$ be a prime power and $n$ some positive integer prime to $q$.
Cyclic codes ${\mathcal C}$
are ideals in the finite ring ${\mathcal M}:=\FF _q[X]/(X^n-1)$.
We identify ${\mathcal M}$ with $\FF _q^n$ using the 
classes of $1,X,\ldots , X^{n-1} $ as a basis. Then 
the  multiplication by $X$ acts on ${\mathcal M}$ as a cyclic permutation 
$\sigma   $. 
In particular the eigenvalues of $\sigma   $ on ${\mathcal M}$ 
(or more precisely $ \overline{\FF _{q}} \otimes _{\FF _q} {\mathcal M} =: 
\overline{\FF _{q}} {\mathcal M}$) are all 
$n$-th roots of unity in the algebraic closure of $\FF _q$, say 
the elements of 
$ {\mathcal Z} := \{ \alpha ^u \mid 0\leq u < n \} $ 
for some primitive $n$-th root of unity $\alpha \in \overline{\FF _{q}} $. 

Based on these data there are (at least) three descriptions of
a given cyclic code ${\mathcal C}$. 
\begin{itemize}
	\item The \emph{generator polynomial} $p = p({\mathcal C})$ 
		which is the monic divisor of $X^n-1$ such that 
the classes of 
$ p, Xp , \ldots , X^{d-1} p $ form a basis of ${\mathcal C} $,
where $d$ is the degree of $(X^n-1)/p$. 
\item The \emph{zero set}  $Z({\mathcal C})$ 
	which is the subset of ${\mathcal Z} $ such that 
	$(c_0,\ldots , c_{n-1}) \in {\mathcal C}$, 
	if and only if $\sum_{i=0}^{n-1} c_i z^i = 0 $ 
	for all $z\in Z({\mathcal C}) $. 
	\item The \emph{eigenvalues} $\Theta ({\mathcal C}) $
		which is the set of eigenvalues of $\sigma   $
		in the $\overline{\FF_q} [\sigma   ] $-module 
		$\overline{\FF_q} {\mathcal C} \leq \overline{\FF_q} {\mathcal M} $.
\end{itemize} 

Clearly we may specify a cyclic code by either of the 
three data, which are related according to the following remark. 

\begin{remark} \label{gencyc}
	$\Theta (C)  = {\mathcal  Z} \setminus Z (C) $, $Z (C)  = {\mathcal Z}\setminus \Theta(C) $, and 
	$Z(C)  = \{ z\in {\mathcal Z} \mid p(z) = 0 \} $ 
	where  $p:= p({\mathcal C})$.
\end{remark}

One important feature of cyclic codes is the fact that 
one can read off a lower bound,
the so called BCH bound, 
on the minimum Hamming distance $\dist (\cC )$.

\begin{theorem} (see \cite[Chapter 7, Theorem 8]{MacWilliamsSloane})\label{BCH}
	Let ${\mathcal C} \leq \FF _q^n$ be a cyclic code.
	Assume that there is some primitive $n$-th root of unity $\alpha \in \overline{\FF _{q}} $ and some $b \geq 0$, $ n\geq \delta \geq 1$ such that 
	$$\{ \alpha ^b , \alpha ^{b+1} ,\ldots , \alpha ^{b+\delta-2} \} 
	\subseteq Z({\mathcal C}) .$$
	Then the minimum Hamming distance $\dist({\mathcal C} )$ 
	of ${\mathcal C}$ is 
	at least $ \delta $. 
\end{theorem}

For any ring $R$ 
the \emph{extended code} of a code ${\mathcal C} \leq R^n $ 
is defined as the code 
$$\{ (c_1,\ldots ,c_n , -\sum _{i=1}^n c_i ) \mid (c_1,\ldots , c_n) \in 
{\mathcal C} \} \leq R^{n+1} .$$
The projection on the first $n$ coordinates is an isomorphism 
between the extended code and the code. 
For cyclic codes, 
one extends  the action of $\sigma   $ to the $n+1$ coordinates by 
$\sigma   (n+1) = n+1 $; then the isomorphism above 
is an $R[\sigma ]$-module isomorphism, in particular 
for codes over fields, the eigenvalues of $\sigma   $ on 
${\mathcal C}$ and its extended code coincide.

\subsection{Chains of cyclic codes and cyclic codes over chain rings}\label{ccc}

Let $q=p^f$ be some power of a prime $p$, $m\in \NN $ and 
$R:= GR(p^m,f) $ denote the Galois ring with $R/pR \cong \FF_q$ 
and characteristic $p^m$. 
Let $n\in \NN $ be not divisible by $p$. 
Then the polynomial 
$$X^n-1  = f_1 f_2 \cdots f_s $$ is a 
product of pairwise distinct monic irreducible polynomials $f_j \in \FF _q[X]$. 
By Hensel's lemma (see also \cite{KL} for a more specific reference)
there are unique monic irreducible 
polynomials $F_j \in R [X ]$ such that 
$$X^n-1 = F_1 F_2 \cdots F_s \in R[X] \mbox{ and } F_j \pmod{p} = f_j .$$
Any chain 
$$(\cC_\star ) \ :  \ \cC _0 = (p_0 ) \subseteq \cC _1 = (p_1) \subseteq \ldots \subseteq 
\cC _{m-1} = (p_{m-1}) \leq \FF_q[X]/(X^n-1) \cong \FF _q^n $$
of cyclic codes is given by a sequence of generator polynomials
$$p_{m-1} \mid  p_{m-2} \mid \ldots \mid p_{1} \mid p_0 \mid (X^n-1) \in \FF _q[X] .$$
Let $P_j \in R[X] $ be the monic divisor of $X^n-1$ that lifts $p_j$. 
Then we define the lift of $(\cC _\star )$ to be the ideal 
$$\widehat{(\cC _\star )} := ( p^j P_j \mid j=0,\ldots , m-1 ) \leq R[X]/(X^n-1) \cong R^n.$$
We can recover the sequence $(\cC_\star )$ from $\widehat{(\cC _\star )} $ by defining 
$\widehat{(\cC _\star )} _j := \widehat{(\cC _\star )} \cap p^j R^n .$
Then
\begin{align} \label{quotC}
\cC _j = \{ (c_1 + pR ,\ldots , c_n + pR ) \mid (p^j c_1,\ldots , p^j c_n) \in \widehat{(\cC _\star) }_j \} \cong 
\frac{\widehat{(\cC _\star) }_j}{\widehat{(\cC _\star) }_{j+1}}
\end{align}
Hence we conclude

\begin{remark}\label{chainbijection} 
Cyclic codes in $R^n$ 
are in bijection to the chains of length $m$ of 
cyclic codes in $\FF_q^n$.
\end{remark}

As before we denote by $\sigma $ the cyclic shift induced by multiplication 
by $X$ on $\FF_q[X]/(X^n-1)$ and on $R[X]/(X^n-1)$. Then
$\FF_q[\sigma ] \cong \FF_q[X]/(X^n-1)$ is a semisimple algebra.

\begin{lemma}  \label{eigenquotgen}
	Assume that we are given two sequences  $(\cC _{\star }) : \ (\cC _i)_{i=0}^{m-1}$ 
	and $(\cD _{\star }) : \ (\cD _i)_{i=0}^{m-1}$
	of cyclic codes 
such that $$\cC _i \subseteq \cD _i \subseteq \cC_{i+1}  $$ 
for all $i$. 
Then 
$$p \widehat{(\cD _\star )}  \subseteq 
\widehat{(\cC _\star )} \subseteq \widehat{(\cD _\star )} \subseteq R^n $$ 
and 
for all $j=0,\ldots , m-1 $
$$\frac{\widehat{(\cD _\star)}_j}{\widehat{(\cC _\star)}_j} \cong 
\frac{\cD_j}{\cC _j}\oplus \frac{\cD_{j+1}}{\cC _{j+1}} \oplus\cdots \oplus \frac{\cD_{m-1}}{\cC_{m-1}}$$
as $\FF_q[\sigma ]$-modules.
\end{lemma}

\begin{proof}
We first note that 
$p\widehat{(\cD _\star )} = \widehat{(\cD^{(1)} _\star )} $ 
where 
$\cD^{(1)}_0 =   \{ 0 \} $ and $\cD^{(1)}_i = \cD_{i-1} $ for 
$i=1,\ldots m-1 $. 
As $\cD_{i-1} \subseteq \cC_i $ we conclude that 
$p \widehat{(\cD _\star )}  \subseteq 
\widehat{(\cC _\star )} $. 
In particular $\widehat{(\cD _\star )}/\widehat{(\cC _\star )} $ is an
$\FF _q[\sigma ]$-module. 
As this algebra is semisimple, all modules are semisimple and it is 
enough to compare composition factors. 
For $0\leq j<m$ consider the $R[\sigma ]$-module epimorphism
$$\varphi _j : p^j R^n \to \FF_q^n \mbox{  defined by }
(p^j c_1,\ldots , p^j c_n) \mapsto (c_1 + pR ,\ldots , c_n + pR ) .$$
The kernel of $\varphi _j $ is $ p^{j+1} R^n$. 
We get 
$$\varphi _j ( \widehat{(\cD _\star )} _j) = \cD _j \mbox{ and } 
\varphi _j ( \widehat{(\cC _\star )} _j) = \cC _j.$$ 
As $p^{j+1} R^n \cap \widehat{(\cD _{\star })}_j = \widehat{(\cD _{\star })}_{j+1} $ and 
$p^{j+1} R^n \cap \widehat{(\cC _{\star })}_j = \widehat{(\cC _{\star })}_{j+1} $
the $\FF_q[\sigma ]$ modules 
${\widehat{(\cD _\star)}_j}/{\widehat{(\cC _\star)}_j}$ and 
$\cD_j/\cC_j \oplus {\widehat{(\cD _\star)}_{j+1}}/{\widehat{(\cC _\star)}_{j+1}}$
have the same composition factors. 
So the lemma follows using induction.
\end{proof}

For chains $(\cC _{\star })$ of extended  cyclic codes, we first lift the
cyclic codes and then extend the lifted code.
The lifted extended code is again denoted by $\widehat{(\cC _{\star })}$. Then 
Remark \ref{chainbijection} and Lemma \ref{eigenquotgen} hold accordingly.

\subsection{Lattices: Construction $\Dcyc$} \label{constAD}

Given a chain of binary codes one may apply Construction D to obtain 
a lattice with a good bound on its minimum 
(see \cite[Chapter 8, Section 8]{SPLAG}).
Construction D,  however, depends on the choice of a
suitable basis and hence might not preserve automorphisms. 
For chains of cyclic codes and extended cyclic codes we may first apply the methods
of Section \ref{ccc} to obtain a cyclic or 
extended cyclic code over $R=\ZZ/p^m\ZZ $ and then 
apply Construction A to this code. 
This construction allows to imitate the proof in \cite{BarnesSloane} 
to obtain good bounds on the minimum of the lattice. 

We keep the notation of the previous section, 
assume that $q=p$ is a prime, so $R=\ZZ/p^m\ZZ $,  and put $N$ to be one
of $n$ (cyclic codes) or $n+1$ (extended cyclic codes).
Additionally we fix an orthogonal basis 
$$( b_i \mid 1\leq i \leq N  ) \mbox{ of } \RR ^{N} \mbox{ with } 
(b_i,b_i ) = p^{-m } \mbox{ for } i = 1,\ldots ,N .$$
We put $\Omega := \langle b_i \mid  1\leq i \leq N   \rangle _{\ZZ }$
to be the lattice spanned by this orthogonal basis and denote by 
$\Phi : \Omega / p^m \Omega \to R^N $ the canonical isomorphism. 

\begin{definition}\label{33}
	Construction $\Dcyc $ associates to 
a chain 
$$(\cC _\star ) \ :  \ \cC _0 \subseteq \cC _1  \subseteq \ldots \subseteq 
 \cC _{m-1} \subseteq  \FF _p^N $$ of cyclic codes or extended 
cyclic codes the lattice 
$${\mathcal L} (\widehat{(\cC _\star )}) :=  \Phi ^{-1} (\widehat{(\cC _\star )}) = 
\{ \sum _{i=1}^N a_i b_i \in \Omega \mid (a_1+p^m \ZZ  ,\ldots , a_N+p^m \ZZ  ) 
\in \widehat{(\cC _\star )} \} .$$
\end{definition}

The lattice ${\mathcal L} (\widehat{(\cC _\star )}) $ obtained by 
construction $\Dcyc $ satisfies
$p^m \Omega \subseteq {\mathcal L} (\widehat{(\cC _\star )}) \subseteq \Omega $ 
and is invariant under the cyclic permutation $\sigma $ of the
basis vectors $(b_i  \mid 1\leq i \leq N )$.

\begin{lemma}  \label{eigenquotgenlat}
	  Given two sequences  $(\cC _{\star }) : \ (\cC _i)_{i=0}^{m-1}$
	  and $(\cD _{\star }) : \ (\cD _i)_{i=0}^{m-1}$ of cyclic or
	  extended cyclic codes such that 
	  $\cC _i \subseteq \cD _i \subseteq \cC_{i+1} $ for all $i$. 
	  Then we have the following isomorphisms of  $\FF _p[\sigma ]$ modules:
	     $$\frac{\cL(\widehat{(\cD _\star)})}{\cL(\widehat{(\cC _\star)})} 
	       \cong \frac{\widehat{(\cD _\star)}}{\widehat{(\cC _\star)}} 
	         \cong \frac{\cD_0}{\cC _0}\oplus \frac{\cD_1}{\cC _1} \oplus\cdots \oplus \frac{\cD_{m-1}}{\cC_{m-1}}.$$
	 \end{lemma}

 \begin{proof}
 Both lattices $\cL(\widehat{(\cD _\star)})$ 
 and $\cL(\widehat{(\cC _\star)})$ contain $p^m \Omega $ 
 so  
 $$\frac{\cL(\widehat{(\cD _\star)})}{\cL(\widehat{(\cC _\star)})} 
 \cong \frac{\cL(\widehat{(\cD _\star)})/p^m\Omega}{\cL(\widehat{(\cC _\star)})/p^m\Omega }  \cong \frac{\widehat{(\cD _\star)}}{\widehat{(\cC _\star)}} .$$
 The second isomorphism is from Lemma \ref{eigenquotgen} putting $j=0$.
	 \end{proof}

\begin{proposition}\label{detLL} 
	The determinant of a Gram matrix of ${\mathcal L} (\widehat{(\cC _\star )}) $
	is $\det ({\mathcal L} (\widehat{(\cC _\star )})) = p^d $ with
$$d = m N - 2\sum_{i=0}^{m-1} \dim(\cC_i) .$$
\end{proposition}

\begin{proof}
	Put $L:={\mathcal L} (\widehat{(\cC _\star )}) $ and 
	for $0\leq j \leq m$ put
	$L_j := \Phi ^{-1} (\widehat{(\cC _\star )}_j ) = L \cap p^j \Omega $.
	Then clearly all the $L_j$ are $\sigma $ invariant sublattices of $\Omega $,
	$L_0 = L$ and $L_m=p^m\Omega $.
	Furthermore by Equation \eqref{quotC} 
	$$L_{j} / L_{j+1} \cong \widehat{(\cC _\star )}_j / \widehat{(\cC _\star )}_{j+1} 
	\cong \cC _j \mbox{ as } \FF _p[\sigma ] \mbox{ modules.} $$
	To compute the determinant of $L$ we compute the index
	$$|L/p^m\Omega |  = \prod _{j=0}^{m-1} |L_j/L_{j+1} | = 
	\prod _{j=0}^{m-1} |\cC _j |  = p^{\sum_{j=0}^{m-1} \dim(\cC_j)} .$$
	Therefore we find
	$$ d = 
	\log_p(\det(L)) = \log_p(\det(p^m\Omega ) ) - 2 \log_p(|L/p^m\Omega | ) =
	m N - 2 \sum_{j=0}^{m-1} \dim(\cC_j) .$$
\end{proof}

The new Construction $\Dcyc $ allows to prove the same bound for the
minimum of the lattice as Construction D.
To state this bound for arbitrary primes $p$ recall that the 
\emph{Euclidean weight}  of $c=(c_1,\ldots , c_N)\in \FF _p^N $ 
is 
$$w_E(c) := \min \{ \sum _{i=1}^N a_i^2 \mid a_i \in \ZZ , a_i+p\ZZ= c_i 
\mbox{ for } i=1,\ldots , N \} .$$ 
Then $\dist _E(\cC ) := \min \{ w_E(c) \mid 0\neq c\in \cC \} $
is the \emph{Euclidean distance} of the code $\cC \leq \FF _p^N $.
Note that $\dist _E(\cC ) = \dist (\cC )$ is the 
usual Hamming distance if $p=2$ or $p=3$.

\begin{theorem}\label{constD}
	Let $(\cC _\star )  $ be as in Definition \ref{33}. 
	Assume moreover that there is $\gamma \geq 1 $ 
	 such that 
	 $\dist _E (\cC _i ) \geq p^{2m-2i}/\gamma $  for all 
	 $0\leq i \leq m-1$. 
	 Then  $\min ({\mathcal L}(\widehat{(\cC_\star )}) )\geq p^m/\gamma .$
\end{theorem} 

\begin{proof}
We keep the notation of the proof of Proposition \ref{detLL}. 
Let $0\neq x\in L$  and let $j$ be maximal 
such that $x\in p^j \Omega $. 
If $j<m$ then $x\in L_j$ and 
$x = p^j y = p^j \sum_{i=1}^N y_i b_i $ with $y_i\in \ZZ $
such that 
$$0 \neq \overline{y} := (y_1 + p\ZZ ,\ldots , y_N + p\ZZ ) \in \cC _j .$$ 
As $\dist _E (\cC_j )  \geq p^{2m-2j} /\gamma $, 
we have $\sum _{i=1}^N y_i^2 \geq p^{2m-2j} /\gamma $ so 
$$(x,x) = p^{2j} (y,y) \geq p^{2j} \frac{p^{2m-2j}}{\gamma }  (b_1,b_1) = 
\frac{p^{2m}}{p^m\gamma } = \frac{p^m}{\gamma } .$$
If $j\geq m$ then $x\in p^m \Omega $, so $(x,x) \geq p^m$.
\end{proof}

\section{Setup and some notation}\label{not} 

Throughout the rest of the paper 
we fix $m\in \ZZ _{>0 }$ and consider codes of length $2^{2m}$ 
and lattices of dimension $2^{2m}$. 
We index our basis by 
the elements of 
${\mathcal V} := \FF _2^{2m} $. 
In particular  binary codes of length $2^{2m}$ 
will be considered as subspaces of the space of functions 
$\FF _2^{{\mathcal V}} : = \{ f: {\mathcal V} \to \FF _2 \} $. 
For any $f\in \FF _2^{{\mathcal V}}$ the 
support of $f$ is $\supp (f) := \{ v\in {\mathcal V} \mid f(v) \neq 0 \} $. 
If $S= \supp (f) $, then clearly $f= \chi _S $ is the 
{\em characteristic function} of $S\subseteq {\mathcal V}$ defined by 
$$\chi _S : {\mathcal V} \to  \FF _2, 
			 v\mapsto \left\{ \begin{array}{cc} 
					                         1 & v\in S \\ 0 & v\not\in S
					                         \end{array} \right. $$

The affine group $\Aff ({\mathcal V}) := {\mathcal V} : \GL({\mathcal V})$
acts on $\FF _2^{{\mathcal V}}$ by permuting the  elements of ${\mathcal V}$.
The Reed-Muller codes from  Definition \ref{RMdef} below are invariant under
$\Aff ({\mathcal V})$.
This invariance is used to view the Reed-Muller codes as extended cyclic
codes.
To this aim we fix a ``Singer-cycle'' 
$$\sigma \in \GL({\mathcal V}) \leq \Aff ({\mathcal V}),$$ 
i.e. an element of order $2^{2m}-1$ permuting the
non-zero elements of ${\mathcal V}$ transitively.
The element $\sigma $ is not unique, even up to conjugacy in $\GL({\mathcal V})$.
Any such $\sigma $ gives rise to an identification of ${\mathcal V}$ with
the field of $2^{2m}$ elements. The eigenvalues of the action of
$\sigma $ as an element of $\GL ({\mathcal V})$ are the elements of
 $$\{ \zeta , \zeta ^2 ,\zeta ^4, \ldots ,\zeta ^{2^{2m-1}} \}$$  for
 a certain
  primitive $(4^m-1)$st root of unity $\zeta \in \overline{\FF } _2 $
  which we fix for the rest of the paper. 

For later use we will fix a vector space structure 
of ${\mathcal V}$ over $\FF _4$ that is defined by $\sigma $. 
	To this aim define
	$\omega := \zeta ^{(4^m-1)/3} $ to be a primitive third 
	root of unity in the algebraic closure of $\FF _2$ 
	(i.e. a primitive element of $\FF _4$).

\begin{remark} \label{eta}
	Let $\eta := \sigma ^{(4^m-1)/3 } \in \GL ({\mathcal V})$.  
	For $v\in {\mathcal V} $ we put $\omega v := \eta (v) $.
	This turns ${\mathcal V} \cong \FF _2^{2m}$ into an
	$m$-dimensional vector space 
	${\mathcal V}_{\FF _4} \cong \FF _4^{m}$ over the field $\FF _4 = \{ 0,1,\omega ,\omega ^2 \}$.
As $\sigma $ commutes with $\eta $, the element $\sigma $ 
acts $\FF _4$-linearly on ${\mathcal V}_{\FF _4}$, so 
$$\sigma \in \GL ({\mathcal V} _{\FF _4})  \leq 
\Aff ({\mathcal V}_{\FF 4}) \cong \FF _4^{m} : \GL _{m} (\FF _4) .$$
Identifying the $\FF _4$-space ${\mathcal V}_{\FF _4}$ with the 
$\omega $-eigenspace of $\eta $ we compute the eigenvalues 
of $\sigma $ on ${\mathcal V}_{\FF _4} \cong \FF _4^m$ as 
$\zeta , \zeta ^4, \ldots , \zeta ^{4^{m-1}} $.
\end{remark}

The following notation will be used throughout the paper. 

\begin{notation}\label{wt2}
	\begin{itemize}
		\item[(a)]
    Any $0\leq u \leq 4^m-1 $ has a unique expression as
     $u= \sum _{i=0}^{2m-1} u_i 2^i $ with $u_i \in \{0,1\} $.
     Then the 2-weight of $u$ is $$\wt_2(u) := |\{ i \in \{ 0,\ldots , 2m-1 \} \mid u_i = 1 \}|
   = \sum _{i=0}^{2m-1} u_i \in \ZZ _{\geq 0}  .$$
   We also define
$$O(u):= | \{ i \in \{ 0, \ldots , m-1 \} \mid  u_{2i+1} = 1 \} | 
\mbox{ and } 
E(u):= | \{ i \in \{ 0, \ldots , m-1 \} \mid u_{2i} = 1 \} | .$$
   \item[(b)]
	For $-1\le r<2m$ we put
$$ Z_r  :=  \{ \zeta^u \mid 0<u\leq 4^m-1, \wt_2(u)\le 2m-1-r \} . $$
\item[(c)]
	For $0\le r\le 2m$ let
  $$ \Theta ^{(r)} 
  :=\{ \zeta ^u \mid 0\leq u \leq 4^{m}-1 , \wt_2(u) = 2m-r \}.  $$
  So $\Theta ^{(0)} = \Theta ^{(2m)} = \{ 1 \} $.  
\item[(d)] 
$M_r := 
\begin{cases}
M_+ := \{ 0\leq k \leq m \mid m-k \mbox{ even } \} & \mbox{ if } r \mbox{ is even } \\ 
M_- := \{ 0\leq k \leq m \mid m-k \mbox{ odd } \} & \mbox{ if } r \mbox{ is odd.} 
\end{cases}
$
  \item[(e)]
For $0\leq k \leq m$ we put 
$$\Theta_k:=\{ \zeta ^u \mid 0\leq u\leq 4^{m}-1, \ |O(u) - E(u)| = m-k \}. $$
\item[(f)]
 Finally, for $0\leq r\leq 2m$ and  $k\in M_r$,  we  define
 $$\Theta ^{(r)}_{k} := \{ \zeta ^u \mid 0\leq u \leq 4^{m}-1, \wt_2(u) = 2m-r, 
|O(u)-E(u)| =m-k \}  = \Theta ^{(r)} \cap \Theta _k .$$
Obviously $\Theta ^{(r)} \cap \Theta _k = \emptyset $ if $k\not\in M_r $.
\end{itemize}
\end{notation}

\begin{lemma} \label{cardTh}
	Let 
$0\leq r\leq 2m$ and  $0\leq k \leq m$. 
\begin{itemize}
\item[(a)]
$|\Theta ^{(r)} | = { 2m \choose r } $. 
\item[(b)]
$|\Theta _k | =  
\begin{dcases} 2 {2m \choose k} & \textup{ if } k < m \\ 
{2m \choose m} -1  & \textup{ if } k = m.
\end{dcases}
$
\item[(c)]
		If $k\in M_{r} $ we have
$$ |\Theta ^{(r)}_{k} | =  
\begin{dcases}
  2{m\choose (m-r+k)/2}{m\choose (k+r-m)/2} &\textup{ if } k<m\\
  {m \choose r/2}^2 &\textup{ if } m=k
\end{dcases}
$$
where we put  ${a \choose b} : = 0 $ if $b < 0 $. 
\end{itemize}
\end{lemma}

\begin{proof}
	(a) is clear and to see (b) 
let $0\leq u\leq 4^{m}-1$ be such that
$O(u) - E(u) = m-k$. Write $u= \sum_{i=0}^{2m-1} u_i 2^i$
with $u_i \in \{ 0,1 \}$ and
  define
$$I :=\{ i \in \{ 0,\ldots , 2m-1 \} \mid
i \mbox{ even  and } u_i =1 \mbox{ or } i \mbox{ odd and } 
u_i = 0 \} .$$
Then $|I| = E(u)+ (m-O(u)) = E(u)-O(u) +m = m - (m-k) = k $.
So $X_k:=\{ u \in \{ 0,\ldots , 4^m-1 \}  \mid O(u) - E(u) = m-k \} $
  is in bijection with the $k$-element subsets
  $I\subset \{ 0,\ldots , 2m-1 \} $ and hence
  has ${2m\choose k}$ elements.
  $X_k$ contains $0$ and $4^m-1$ if and only if $k=m$ 
  so $|\Theta _m| = |X_m| -1 $  and 
  $|\Theta _k| = 2|X_k|$ if $k<m$.
  \\
 (c) follows by a straightforward counting argument. 
\end{proof}

\section{Reed-Muller codes and related extended cyclic codes}

\subsection{Binary Reed-Muller codes of length $2^{2m}$} \label{RMcyc}

\begin{definition}\label{RMdef}
   For $0\leq r \leq 2m $ let
  $$\cR( r,2m) := \langle \chi _{a+{\mathcal U}} \mid a\in {\mathcal V}, \ 
  {\mathcal U}\leq {\mathcal V} \mbox{ a subspace of dimension }  \dim ({\mathcal U}) = 2m-r \rangle $$
   denote the $r^{\textup th}$ order \emph{binary Reed-Muller code} of length $2^{2m}$.
   \\
   To simplify notation we put $\cR (-1,2m) := \{ 0 \} $. 
\end{definition}

Some well known properties of the Reed-Muller codes are collected
in the following remark. 

\begin{remark}\label{RMdetails}
	 \begin{itemize}
 \item[(a)] 
      $\FF _2^{2^{2m}} = \cR( 2m,2m) \supset \cR( 2m-1,2m)\supset \ldots \supset 
        \cR( 1,2m) \supset \cR( 0,2m) = \langle {\bf 1 } \rangle $.
\item[(b)] The dimension of $\cR( r,2m) $ is
     $\dim (\cR( r,2m)) = \sum _{\ell =0}^{r} {{2m}\choose{\ell}}  $.
    \item[(c)] The dual code is  $\cR( r,2m) ^{\perp } = \cR( 2m-r-1,2m)$.
    \item[(d)]
 For the minimum distance we have
 $\dist (\cR( r,2m)) = 2^{2m-r}  $ where $0\leq r\leq 2m$.
Moreover the minimum weight vectors in $\cR( r,2m)$
	are the elements of
	$$\{ \chi _{a+{\mathcal U}} \mid a\in {\mathcal V}, \ 
	{\mathcal U}\leq {\mathcal V}, \ \dim ({\mathcal U}) = 2m-r \} .$$
 \end{itemize}
\end{remark}

To define a convenient basis of the Reed-Muller codes 
we fix a basis $(v_1,\ldots , v_{2m}) $ of $\cV $ and put
$${\mathcal T}_r:= \{ \cU \leq \cV \mid \cU = \langle v_i \mid i\in I \rangle _{\FF _2} 
\mbox{ where } I \subseteq \{ 1,\ldots , 2m \} \mbox{ with }  |I| = r \} .$$
Then we find 
\begin{proposition} (cf. \cite[p. 51]{BarnesWall}) \label{basisRM}
	For $0\leq r\leq 2m $ the set 
	$$\{ \chi _{\cU } \mid \cU \in {\mathcal T}_s, 2m-r\leq s\leq 2m \} $$
	is a basis of $\cR (r,2m) $ and  the classes of
	$$\{ \chi _{\cU } \mid \cU \in {\mathcal T}_{2m-r} \} $$
	form a basis of $\cR(r,2m)/\cR(r-1,2m) $.
\end{proposition}

The affine group $\Aff ({\mathcal V}) := {\mathcal V} : \GL({\mathcal V})$
acts on $\FF _2^{{\mathcal V}}$ by permuting the  elements of ${\mathcal V}$.
As affine transformations preserve the set of affine subspaces of a given
dimension, the Reed-Muller codes are invariant under
$\Aff ({\mathcal V})$.
In particular the Singer-cycle $\sigma $ defined in  Section \ref{not} 
is an automorphism of all the Reed-Muller codes from Definition \ref{RMdef} 
and these codes are extended cyclic codes as given in the
following remark.

\begin{remark} (cf. \cite[Chapter 13, Theorem 11]{MacWilliamsSloane})\label{RM}
	For $-1\le r<2m$, define  $\cR(r,2m)^*$ to be the length $4^{m}-1$ binary cyclic code with zeros 
	$Z(\cR(r,2m)^* ) = Z_r $ where $Z_r$ is as in Notation \ref{wt2} (b). 
The extended code of $\cR( r,2m)^*$ is the $r^{\textup{th}}$ order 
binary Reed-Muller code $\cR( r,2m)$. Note that $\cR (2m,2m) = \FF _2^{2^{2m}}$ is the universe code
which is not an extended cyclic code. 
\end{remark}

Applying Remark \ref{gencyc} we obtain the eigenvalues of $\sigma $ on $\cR( r,2m)/\cR( r-1,2m) $:

\begin{proposition} \label{eigenvalsR}
  For $0\leq r \leq 2m$ the eigenvalues of $\sigma $ on
  $$\cR( r,2m)/\cR( r-1,2m) $$ 
  are exactly the elements in 
  $ \Theta ^{(r)}$ from Notation \ref{wt2} (c). 
\end{proposition}

\subsection{Extended cyclic codes sandwiched between
 Reed-Muller codes} \label{RMgen}

In this section we construct some new 
extended cyclic codes  
that are invariant under  $\Aff ({\mathcal V}_{\FF _4}) $.
We use the notation introduced in Section \ref{not}.

\begin{definition} 
Let $0\leq r<2m $ and $I\subset M_{r}$ be given. Put 
$$Z_{r,I} := Z_{r-1} \setminus (\bigcup _{k\in I} \Theta ^{(r)}_{k} ) .$$
Note that $Z_r \subseteq Z_{r,I} \subseteq Z_{r-1} $. 
Then let 
${\mathcal C}(r,I,2m)^* \leq \FF _2^{2^{2m}-1} $ be the cyclic code with zero set $Z_{r,I}$ 
and ${\mathcal C}(r,I,2m) \leq \FF _2^{2^{2m}} $ the extended code of ${\mathcal C}(r,I,2m)^*$. 
Also we define
$$
\cC(2m,I,2m) =
\begin{cases}
  \cR(2m-1,2m) & \textup{if } m\not\in I\\
  \cR(2m,2m)=\FF_2^{2^{2m}} &\textup{otherwise.}
\end{cases}
$$
\end{definition}

Comparing zero sets we immediately get the following remark.

\begin{remark}\label{CIdetails}
\begin{itemize}
\item[(a)] $\cR (r-1,2m) \subseteq {\mathcal C}(r,I,2m) \subseteq \cR (r,2m ) $.
\item[(b)] $\cR (r-1,2m) = {\mathcal C}(r,\emptyset,2m) $.
\item[(c)] $\cR (r,2m) = {\mathcal C}(r,M_r,2m) $.
\item[(d)] If $I\subseteq J \subseteq M_r$ then 
${\mathcal C}(r,I,2m) \subseteq {\mathcal C}(r,J,2m) $.
\item[(e)] The eigenvalues of 
	$\sigma $ on 
	${\mathcal C}(r,I,2m) /\cR (r-1,2m)  $ are 
	exactly the elements in 
	$\bigcup_{k\in I} \Theta ^{(r)}_{k} $. 
\item[(f)] $\dim ({\mathcal C}(r,I,2m) ) = \dim (\cR(r-1,2m) )
	+ \sum _{k\in I} |\Theta ^{(r)}_{k} | 
	= \sum _{\ell =0}^{r-1} { 2m \choose \ell }
	+ \sum _{k\in I} |\Theta ^{(r)}_{k} | $
	\\ where $|\Theta ^{(r)}_{k}| $ can be obtained from
	Lemma \ref{cardTh} (c).
	\end{itemize}
\end{remark}

The next proposition can be obtained from the arguments in Section \ref{last} 
as $\Aff ({\mathcal V}_{\FF _4}) \subseteq \Aff ({\mathcal V}) \cap {\mathcal U}_{m} $ 
where ${\mathcal U}_m$ is defined in Definition \ref{Um}. 
It also follows from \cite[Theorem 5.5]{Abdukhalikov}.

\begin{proposition}
	For all $0\leq r \leq 2m$ and all $I\subseteq M_r $ 
	the automorphism group of ${\mathcal C}(r,I,2m)$ 
	contains  $\Aff ({\mathcal V}_{\FF _4} ) $.
\end{proposition}

Applying the BCH bound, we find the following lower bounds on
the minimum distance of the codes ${\mathcal C}(r,I,2m)$.

\begin{theorem}\label{thm:distance}
	Let $1\leq r \leq 2m-1$ and  $I\subseteq M_r $. Then
	$$\dist ({\mathcal C}(r,I,2m) ) 
      \begin{dcases}
	      =2^{2m-r+1} = \dist (\cR (r-1,2m))   & \textup{ if } \{ m, m-1 , m-2\} \cap I = \emptyset \\
	      \ge 2^{2m-r} = \dist (\cR (r,2m)) & \textup{ if } \{ m, m-1 \}  \cap  I \neq \emptyset  \\
	      \ge 3\cdot 2^{2m-r-1}  & \textup{ if }  \{ m , m-2 \} \cap   I  = \{ m-2 \} 
      \end{dcases}
  $$
\end{theorem}

\begin{proof}
Clearly 
$$2^{2m-r} = \dist (\cR (r,2m)) \leq 
\dist ({\mathcal C}(r,I,2m) ) \leq \dist (\cR (r-1,2m)) =2^{2m-r+1} .$$ 
To obtain the minimum distance of $\cR (r-1,2m) $ one uses the 
BCH bound (cf. Theorem \ref{BCH}), showing that 
$$Z:=\{ \zeta ^u \mid 0< u < 2^{2m-r+1}-1 \} $$ 
are in the zero set of $\cR (r-1,2m)^* $ as all these exponents $u$ 
have 2-weight $\leq 2m-r$. 
The zero set of ${\mathcal C}(r,I,2m)^*$ contains all these 
$\zeta ^u \in Z$ with $\wt_2(u) < 2m-r$ and those $\zeta ^u \in Z$ 
with $\wt_2(u) = 2m-r $
such that $|E(u) - O(u) | = m-k $ with $k\not\in I$. 
So let $0< u < 2^{2m-r+1} -1 $ be such that 
$\wt_2(u) = 2m-r $. Then 
$u = \sum _{i=0}^{2m-r} u_i 2^i $ with $u_i=0$ for exactly one $i$. 

If $r$ is odd then one easily concludes that $|O(u) - E(u) | = 1$. 
So if $r$ is odd and $m-1 \not\in I$ then $Z$ is in
the zero set of ${\mathcal C}(r,I,2m)^*$, so the BCH bound 
allows to conclude that $\dist({\mathcal C}(r,I,2m)) = \dist(\cR (r-1,2m)).$

If $r$ is even, then $|O(u) - E(u) | \in \{ 0, 2 \} $, showing 
again that $Z \subseteq Z({\mathcal C}(r,I,2m)^*) $ and 
$\dist({\mathcal C}(r,I,2m)) = \dist(\cR (r-1,2m))$ if 
$I\cap \{ m, m-2 \} = \emptyset $. 
The minimal $u$ such that 
$|O(u) - E(u) | = 2$ is $u= 2^{2m-r-1}-1+2^{2m-r} = 3\cdot 2^{2m-r-1} -1 $ so the BCH bound gives 
$\dist ({\mathcal C}(r,I,2m)) \geq 3\cdot 2^{2m-r-1} $
if $I\cap \{ m, m-2 \} =  \{ m-2 \} $. 
\end{proof}

\section{Unitary invariant sandwiched lattices} 

\subsection{The Barnes-Wall construction} \label{constructD}

To construct the Barnes-Wall lattice $\BW_{2m} \leq \RR ^{2^{2m}}$ 
and related lattices we fix an orthogonal basis 
$$( b_v \mid v\in {\mathcal V}) \mbox{ of } \RR ^{2^{2m}} \mbox{ with } 
(b_v,b_v ) = 2^{-m } .$$
We put $\Omega := \langle b_v \mid v\in {\mathcal V} \rangle _{\ZZ }$ 
to be the lattice spanned by this orthogonal basis. 
Then \cite{BarnesWall} constructs the Barnes-Wall lattices  $\BW _{2m}$ and 
its dual $\BW ^{\#} _{2m} $ 
as lattices $L$ with 
$$2^m \Omega \subseteq L \subseteq \Omega $$ by 
scaling the basis of the Reed-Muller codes given 
in Proposition \ref{basisRM}. 

\begin{definition}(\cite[Theorem 3.1]{BarnesWall})\label{BWdef}
	$$\BW_{2m} := \langle 2^{\lfloor \frac{2m-r+1}{2} \rfloor } \sum_{v\in \cU} b_v  \mid 
	\cU \in {\mathcal T}_r , r=0,\ldots , 2m  \rangle _{\ZZ }$$ 
	is the Barnes-Wall lattice of dimension $2^{2m}$ 
	and its dual lattice is given as
	$$\BW_{2m}^{\#}  = \langle 2^{\lfloor \frac{2m-r}{2} \rfloor } \sum_{v\in \cU} b_v  \mid 
	\cU \in {\mathcal T}_r , r=0,\ldots , 2m  \rangle _{\ZZ }.$$ 
\end{definition}
Note that the generators for the lattices in Definition \ref{BWdef}
form a basis of $\BW_{2m}$ and $\BW_{2m}^{\#} $. 

The parameters for  the Barnes-Wall lattices  are
$$\det (\BW_{2m}) = 2^{2^{2m-1}} , \min (\BW_{2m}) = 2^m , 
\BW_{2m}^{\#} / \BW_{2m} \cong \FF _2^{2^{2m-1}} $$
(see \cite{BarnesWall} and \cite{BroueEnguehard}).

The Barnes-Wall construction in Definition \ref{BWdef} is a very specific
variant of Construction D 
applied to the two chains of Reed-Muller codes: 
$$\begin{array}{rclclclclr}
                (\cR _{2\star }) &  :  & 
     \cR( 0,2m)  & \subset  & \cR( 2,2m) & \subset  & \ldots &  \subset  & \cR( 2m-2,2m)  & \mbox{ and }   \\
         (\cR _{2\star -1}) & : &
         \cR( 1,2m) &  \subset &  \cR( 3,2m) & \subset  & \ldots &  \subset 
	          & \cR( 2m-1,2m)  . \end{array} $$
Note that Construction D in general depends on the chosen basis 
adapted to the chain of codes as explained in detail in \cite{Oggier},
where the authors compare Construction D and D' with Forney's Code-Formula
construction. 
Their main result is \cite[Theorem 1]{Oggier} showing that 
Construction D and Forney's Code-Formula construction yield the 
same lattice if and only if the chain of nested binary codes is 
closed under the Schur product. Only then Construction D 
does not depend on the choice of the basis. 

\begin{warning} 
	For $m\geq 4$ then 
	$(\cR _{2\star } ) $ and $(\cR _{2\star -1 })$ are not 
	closed under the Schur product. 
	So in contrast to many remarks in the literature
	(e.g. \cite[bottom of p. 447]{Oggier})
	the lattice constructed by Construction D from 
	these chains of codes  will depend on the chosen basis. 
\end{warning}

\begin{proof}
Recall that the Schur product is a function 
$\FF _2^n \times \FF _2^n \to \FF _2^n $ mapping 
$(c,d) $ to $c*d$ with $(c*d)_i = c_i d_i $. 
By \cite[Section (13.3)]{MacWilliamsSloane} 
$\cR (r,2m) $ is the set of all vectors $f$, where 
$f(v^*_1,\ldots , v^*_{2m} )$ is a Boolean function, which can
be written as a 
polynomial of degree at most $r$ in the symmetric algebra of ${\mathcal V}^*$. 
So $f$ is a linear combination of $\prod _{i\in I } v^*_i $ 
where $I\subseteq \{1,\ldots, 2m \} $, $|I | \leq r$. 
The Schur product of Boolean functions translates into the 
product of polynomials subject to the relations ${v^*_i}^2=v^*_i$ for all $i$. 
If $m\geq 4$ then $v^*_1v^*_2v^*_3v^*_4 $ and $v^*_5v^*_6v^*_7v^*_8$ are in 
$\cR (4,2m)$ but their product has degree 8, hence does not
belong to $\cR(6,2m)$,  the next member of the chain $(\cR _{2\star })$. 
A similar argument also applies to $(\cR _{2\star -1})$, where it
is enough to assume $m\geq 3$. 
\end{proof}

\subsection{Construction $\Dcyc$ for the Barnes-Wall lattices}

By \cite[Theorem 3.2]{BarnesWall} the affine group
$\Aff ({\mathcal V} )$ 
 acts on the lattice 
$\BW_{2m}$ and its dual lattice $\BW_{2m}^{\# } $ by permuting the
basis vectors $(b_v \mid v\in{\mathcal V}) $. 
This  action also preserves the Reed-Muller codes and 
in particular these codes and 
the lattices $\BW_{2m}$ and $\BW_{2m}^{\#}$ are invariant 
  under the cyclic permutation $\sigma $. 
  Hence also their 
  quotients $\BW_{2m} / 2^m\Omega $ and $\BW_{2m}^{\#} /2^m\Omega $ 
  are invariant under $\sigma $. 
  As the sums of the coefficients in the given basis vectors
  of $\BW_{2m}$ and $\BW_{2m}^{\#} $ sum up to a multiple of $2^{m}$ 
  these are extended cyclic codes in 
  $ \Omega /2^m\Omega  \cong 
  (\ZZ/2^m\ZZ )^{2^{2m}} .$
  In the notation of Section \ref{ccc} Remark \ref{chainbijection} 
  hence tells us 
  $$\BW_{2m} / 2^m\Omega \cong \widehat{(\cR _{2\star } )}  
  \mbox{ and } 
  \BW_{2m}^{\#}  / 2^m\Omega \cong \widehat{(\cR _{2\star -1 } )}  .$$ 

  \begin{remark} \label{BWL}
	  $\BW _{2m} = {\mathcal L} (\widehat{(\cR _{2\star } )} ) $ and 
		  $\BW_{2m}^{\#}  = {\mathcal L} (\widehat{(\cR _{2\star -1 } )})$ are the lattices obtained by 
		  Construction $\Dcyc $ from the two chains 
		  of Reed-Muller codes above.
  \end{remark}

\begin{proposition} \label{eigenquot}
	As $\FF _2 [\sigma ] $-modules we have 
	  $$ \BW_{2m}^{\#} /\BW_{2m}  \cong \bigoplus_{r=0}^{m-1} 
	  \cR( 2r+1,2m) / \cR( 2r,2m) $$ 
	  and
	  $$ \BW_{2m}/2\BW_{2m}^{\#} \cong 
	  \cR (0,2m) \oplus  \bigoplus_{r=1}^{m} \cR( 2r,2m)/\cR( 2r-1,2m) .$$ 
	  The eigenvalues of $\sigma $ on 
	  $\BW_{2m}/2\BW_{2m}^{\#} $ are the elements of 
	  $$ {\Theta }^{(+)} := \{ \zeta ^u  \mid 0\leq u < 2^{2m}-1 
	  \mbox{ of even $2$-weight } \} =\bigcup _{r=1}^{m} \Theta^{(2r)} $$
	  where $\zeta ^0 = 1$ occurs with multiplicity $2$ (and the
	  others with multiplicity $1$) in 
	  $\BW_{2m}/2\BW_{2m}^{\#}$
	  and the one on 
	  $ \BW_{2m}^{\#} /\BW_{2m} $ are 
  the elements of $$ {\Theta }^{(-)} := \{ \zeta  ^u  \mid 0 \leq u <2^{2m}-1 
	  \mbox{ of odd $2$-weight } \} = \bigcup _{r=1}^{m} \Theta^{(2r-1)}$$ 
	  each occurring with multiplicity 1.
  \end{proposition}

  \begin{proof}
	  The isomorphism of $\BW_{2m}^{\#} /\BW_{2m}  $ follows directly by
	  applying Lemma \ref{eigenquotgenlat}. 
	  With a variant of this lemma we may also see the isomorphism 
	  of $ \BW_{2m}/2\BW_{2m}^{\#} $, but this may be also seen
	  from the following consideration:
	  We have $\Omega /2\Omega \cong \cR(2m,2m) = \FF_2^{2^{2m}} $ 
	  as $\FF_2[\sigma ]$-modules. 
	  As $\FF _2[\sigma ]$ is semisimple, it is enough to 
	  compare composition factors so the chain of Reed-Muller codes
	  in Remark \ref{RMdetails} (a) 
	  shows that 
	  $$\Omega/2\Omega \cong \bigoplus _{r=0}^{2m} \cR(r,2m)/\cR(r-1,2m) $$
	  (note that $\cR(-1,2m) = \{0 \}$).
	  Now $\BW_{2m}^{\#}$ and $\Omega $ are lattices in the
	  same $\QQ [\sigma ]$-module, so 
	  $\BW_{2m}^{\#} /2 \BW_{2m}^{\#} $ and $\Omega /2\Omega $ have 
	  the same composition factors (see \cite[Theorem 32]{Serre}), therefore
	  $$\BW_{2m}^{\#} /2\BW_{2m}^{\#}  \cong 
	  \BW_{2m}^{\#} /\BW_{2m} \oplus \BW_{2m}/2\BW_{2m}^{\#} 
	  \cong \bigoplus _{r=0}^{2m} \cR(r,2m)/\cR(r-1,2m) $$ 
	  so 
$\BW_{2m}/2\BW_{2m}^{\#} \cong 
	  \cR (0,2m) \oplus  \bigoplus_{r=1}^{m} \cR( 2r,2m)/\cR( 2r-1,2m) .$
	  The eigenvalues are obtained from 
	  Proposition \ref{eigenvalsR}.
  \end{proof}

  \subsection{Admissible sandwiched lattices}

  \begin{definition}\label{sigmasandwich}
	  A $\sigma $ invariant lattice  $\Gamma $ with 
	  $2\BW _{2m} ^{\#} \subseteq \Gamma \subseteq  \BW _{2m} $ 
	  is said to be {\em admissible}, if 
	  either $1$ does not occur as an eigenvalue of $\sigma $ 
	  on $\Gamma / 2\BW _{2m} ^{\#} $ or 
	  it occurs with multiplicity $2$. 
Let 
$${\mathcal L}_{+} := \{ \Gamma \mid 2\BW _{2m} ^{\#} \subseteq \Gamma \subseteq  \BW _{2m} , \sigma (\Gamma ) = \Gamma , \Gamma \mbox{ admissible}  \} $$ 
and 
$${\mathcal L}_{-} := \{ \Lambda \mid \BW _{2m} \subseteq \Lambda \subseteq  \BW _{2m}^{\#} , \sigma (\Lambda ) = \Lambda  \} $$ 
denote the set of $\sigma $ invariant \emph{admissible sandwiched lattices}. 
  \end{definition}

 By definition, the  admissible sandwiched lattices 
 are  in bijection with  the monic factors in $\FF _2[X]$ 
 of the minimal polynomial  of the action of $\sigma $ on 
 $\BW_{2m}^{\#} /\BW_{2m} $ and $\BW_{2m}/2\BW_{2m}^{\#}$, so by Proposition 
 \ref{eigenquot} 
 with the subsets of 
 ${\Theta }^{(-)}$ resp. ${\Theta }^{(+)}$ that are closed under 
 squaring:

 \begin{proposition}\label{labellat} 
	 \begin{itemize}
		 \item[(a)] 
			 Let $S \subseteq {\Theta }^{(+)} $ 
			 be a Frobenius invariant subset, i.e. $s\in S $ if and only 
			 if $s^2\in S$. 
			 Then there is a unique lattice 
			 $\Gamma \in {\mathcal L}_{+} $ such that 
			 the characteristic polynomial of the action of 
			 $\sigma $ on $\Gamma / 2\BW^{\#}_{2m}  $ is 
			 $\prod _{s\in S} (X-s) \in \FF_2[X] $ if 
			 $1\not\in S $ respectively 
			 $(X-1) \prod _{s\in S} (X-s) \in \FF_2[X] $ if 
			 $1\in S$.
		 \item[(b)] 
			 Let $S \subseteq {\Theta }^{(-)} $ 
			 be a Frobenius invariant subset, i.e. $s\in S $ if and only 
			 if $s^2\in S$. 
			 Then there is a unique lattice 
			 $\Lambda \in {\mathcal L}_{-} $ such that 
			 the characteristic polynomial of the action of 
			 $\sigma $ on $\Lambda  / \BW_{2m} $ is 
			 $\prod _{s\in S} (X-s) \in \FF_2[X] $. 
	 \end{itemize}
 \end{proposition}
 
 \subsection{Unitary invariant sandwiched lattices} \label{unitarysand}

 Recall the definition of $M_+$ and $M_-$ in Notation \ref{wt2}. 
  For proper subsets $\emptyset \neq 
  I\subset M_-$  or $\emptyset \neq J\subset M_+$ we put
  $$
  \begin{array}{lclclclclclr}  
	  (\cC _{\star I}) & \ :\  &
	  {\mathcal C}(1, I, 2m) & \subseteq  &
	  {\mathcal C}(3, I, 2m) & \subseteq  &
	   \ldots & \subseteq & 
	  {\mathcal C}(2m-1, I, 2m)  & \mbox{ if } I \subseteq M_-, \\ 
	  (\cC _{\star J}) & \ :\  &
	   {\mathcal C}(0, J, 2m) & \subseteq  &
	   {\mathcal C}(2, J, 2m) & \subseteq  & 
	   \ldots & \subseteq  & 
	   {\mathcal C}(2m-2, J, 2m) & \mbox{ if } m\in J \subseteq M_+, \\ 
	  (\cC _{\star J}) & \ :\  &
	  {\mathcal C}(2, J, 2m) & \subseteq  & 
	  {\mathcal C}(4, J, 2m)&  \subseteq  & 
	  \ldots & \subseteq  & 
	  {\mathcal C}(2m, J, 2m) &  \mbox{ if } m\not\in J \subseteq M_+ . 
  \end{array}
  $$
  Note that for $J\subseteq M_+$ we have
  ${\mathcal C}(2m, J, 2m)  =  \cR (2m,2m) = \FF_2^{2^{2m}} $ if 
  $m\in J$ and ${\mathcal C}(0, J, 2m) = \{ 0\} $ if $m\not\in J$.

  \begin{remark}
	  We will see in Section \ref{last} 
	  that the lattices ${\mathcal L}(\widehat{(\cC _{\star I})} ) $ 
	  and 
	  ${\mathcal L}(\widehat{(\cC _{\star J})} ) $ 
	  constructed
	  from these chains of extended cyclic codes with 
	  Construction $\Dcyc$ 
	  are invariant under the Clifford-Weil group 
	  $${\mathcal U}_m  = \cC _m(4^H_{\bf 1}) \cong 2^{1+4m}_{+} : \GU _{2m}(\FF _4)$$ 
	  associated to 
	  the Type of Hermitian self-dual codes over $\FF _4$ 
	  that contain the all ones vector (see \cite[Proposition 7.3.1]{NebeRainsSloane}).
	  Therefore we call the lattices 
	  ${\mathcal L}(\widehat{(\cC _{\star I})} ) $ and
	  ${\mathcal L}(\widehat{(\cC _{\star J})} ) $,
	  obtained by applying Construction $\Dcyc $  to the 
	  chain of codes  $(\cC_{\star I})$ and $(\cC_{\star J})$ 
	  above 
	  \emph{unitary invariant sandwiched lattices}.
  \end{remark}

  \begin{theorem}\label{sandlatdetails}
	  \begin{itemize}
		  \item[(a)] 
	  If  $\emptyset \neq I\subset M_-$  then
	  ${\mathcal L}(\widehat{(\cC _{\star I})} ) \in {\mathcal L}_{-} $ 
	  and the eigenvalues of $\sigma $ on 
	  ${\mathcal L}(\widehat{(\cC _{\star I})} )/\BW_{2m} $ are the elements
	  of $\bigcup _{k\in I} \Theta _k$.
	  We get 
	  $$\log_2(\det({\mathcal L}(\widehat{(\cC _{\star I})} ))) =
	  2^{2m-1} - 4 \sum _{k\in I} {2m \choose k } .$$
	  If $m-1 \not\in I$, then 
	  $$\min ({\mathcal L}(\widehat{(\cC _{\star I})}  ) = \min (\BW_{2m}) = 2^m .$$
  \item[(b)]
	  For $\emptyset \neq J\subset M_+$ with $m\in J$ then 
	  ${\mathcal L}(\widehat{(\cC _{\star J})} ) \in {\mathcal L}_{+}$ 
	  and the eigenvalues of $\sigma $ on 
	  ${\mathcal L}(\widehat{(\cC _{\star J})} )/2 \BW_{2m}^{\#}  $ are the elements
	  of $\bigcup _{k\in J} \Theta _k$.
	  We get 
	  $$\log_2(\det({\mathcal L}(\widehat{(\cC _{\star J})} ))) =
	  2^{2m-1} + 4 \sum _{k\in M_+\setminus J} {2m \choose k } .$$
  \item[(c)]
	  For $\emptyset \neq J\subset M_+$ with $m\not\in J$ then 
	  $2{\mathcal L}(\widehat{(\cC _{\star J})} ) \in {\mathcal L}_{+}$ 
	  and the eigenvalues of $\sigma $ on 
	  $2{\mathcal L}(\widehat{(\cC _{\star J})} )/ 2\BW_{2m}^{\#}  $ are the elements of $\bigcup _{k\in J} \Theta _k $.
	  We get 
	  $$\log_2(\det(2{\mathcal L}(\widehat{(\cC _{\star J})} ))) =
	  2^{2m-1} + 4 \sum _{m\neq k\in M_+\setminus J} {2m \choose k } + 2 {2m \choose m} .$$
	  If, furthermore, $m-2 \not\in J$ then 
	  $$\min (2{\mathcal L}(\widehat{(\cC _{\star J})}  ) = \min (2\BW^{\#}_{2m}) = 2^{m+1} .$$
	  \end{itemize}
  \end{theorem}

  \begin{proof}
     Here we only present the proof of (a), as (b) and (c) can be proved very similarly.
    For (a), from Remark~\ref{BWL} we know that $\BW_{2m}={\mathcal L}(\widehat{(\cR _{2\star})} )$.
    Note that the sequences $(\cC_{\star I})$ and $(\cR_{2\star})$ satisfy the condition of
  Lemma~\ref{eigenquotgenlat}. Hence
    $$
    \frac{{\mathcal L}(\widehat{(\cC _{\star I})} )}{\BW_{2m}}
   \cong \frac{\cC(1,I,2m)}{\cR(0,2m)}\oplus \frac{\cC(3,I,2m)}{\cR(2,2m)}\oplus\cdots\oplus \frac{\cC(2m-1,I,2m)}{\cR(2m-2,2m)}
$$
as $\FF_{2}[\sigma]$-modules. By (e) of Remark~\ref{CIdetails} it follows that the eigenvalues of $\sigma$ on
 ${\mathcal L}(\widehat{(\cC _{\star I})} )/\BW_{2m} $ are the elements
   of $\bigcup _{k\in I} \Theta _k$.
    Now the determinant follows directly by Lemma~\ref{cardTh}.
  As ${\mathcal L}(\widehat{(\cC _{\star I})} )\supseteq \BW_{2m}$, we have
$\min ({\mathcal L}(\widehat{(\cC _{\star I})}  ) \leq \min (\BW_{2m}) = 2^m.$
If $m-1\not\in I$, then by Theorems~\ref{constD} and~\ref{thm:distance},
        $\min ({\mathcal L}(\widehat{(\cC _{\star I})}  ) \geq 2^m.$
   This concludes our proof.
  \end{proof}
 
\section{Automorphism groups}

\subsection{The automorphism group of the Barnes-Wall lattices}
The automorphism groups of the Barnes-Wall lattices 
have been described by Brou\'e and Enguehard 
and independently in a series of papers by Barnes, Wall, Bolt, and 
Room. 

\begin{theorem} (\cite{BroueEnguehard}, \cite[Theorem 3.2]{Wall}) 
	${\mathcal G}_{2m} := \Aut (\BW_{2m}) = 2^{1+4m}_+. O_{4m}^+(2) $.
\end{theorem}

Here $O_{4m}^+(2)$ is the orthogonal group of a quadratic form $q$ 
of dimension $4m$ over $\FF _2$ and Witt defect 0. 
Let ${\mathcal E}_{2m} \cong 2^{1+4m}_+ \leq {\mathcal G}_{2m} $ denote
the maximal normal 2-subgroup of ${\mathcal G}_{2m}$. 
Then $Z:=Z({\mathcal E}_{2m} ) \cong C_2 $ and 
$$q: {\mathcal E}_{2m}/Z \to Z,  xZ \mapsto x^2  $$ 
can be viewed as the $O_{4m}^+(2)$ invariant quadratic form. 
The affine group $\Aff ({\mathcal V})$ acts as orthogonal 
mappings on $\RR ^{2^{2m}}$ by permuting
the basis vectors $(b_v \mid v \in {\mathcal V})$. This action 
stabilizes the Barnes-Wall lattice, so $\Aff ({\mathcal V}) \leq 
{\mathcal G}_{2m} $.
In fact this embedding is made explicit in \cite[Lemma 3.2]{BarnesWall}. 
The additive group of ${\mathcal V}$ can be seen as a maximal
isotropic subgroup $\FF_2^{2m} \leq {\mathcal E}_{2m} $ with respect to
the quadratic form $q $ from above 
and $\GL({\mathcal V})$ is its stabilizer in the orthogonal group of $q$. 
In particular we obtain an explicit
elements $\sigma $ and  $\eta = \sigma ^{(4^m-1)/3}$ 
(from Remark \ref{eta})  in  ${\mathcal G}_{2m} $. 

\begin{definition}\label{Um}
	Define 
	${\mathcal U} _{m}  \leq {\mathcal G}_{2m}$ to be the 
	normaliser in ${\mathcal G}_{2m} $ of 
	${\mathcal E}_{2m}  :\langle \eta \rangle   $.
\end{definition} 

Note that 
$\eta $ defines an $\FF _4$-linear structure on $\FF _2^{4m} $ (similar as in Remark \ref{eta}) 
turning the natural quadratic $O_{4m}^+(2)$-module 
into a Hermitian space over $\FF _4$. 
Then ${\mathcal U}_{m} \cong  {\mathcal E}_{2m}. \GU_{2m}(\FF _4)  $ 
is the extension of ${\mathcal E}_{2m}$ 
by the semi-linear unitary group $\GU_{2m}(\FF _4)  $ of this Hermitian space. 
Intersecting the subgroup 
$\Aff ({\mathcal V})$  of ${\mathcal G}_{2m}$ with  ${\mathcal U}_m $
we find that $\Aff ({\mathcal V}_{\FF _4} ) \leq {\mathcal U}_m$.

One name for ${\mathcal G}_{2m}$ is  Clifford collineation group, 
because  the modules 
$$ \BW_{2m}/2\BW_{2m}^{\#} \cong \FF_2^{2^{2m-1}} \mbox{ and }  
 \BW_{2m}^{\#} /\BW_{2m} \cong \FF_2^{2^{2m-1}}  $$
 are simple modules 
  for the even Clifford algebra.
  In particular $ \BW_{2m}/2\BW_{2m}^{\#} $ and $ \BW_{2m}^{\#} /\BW_{2m}$ are  simple $\FF_2 {\mathcal G}_{2m} $-modules
  (called a spin representation) 
  having ${\mathcal E}_{2m}$ in their kernel. 
  So  ${\mathcal E}_{2m}$ is in the automorphism group of every 
  sandwiched lattice 
  $L \in {\mathcal L}_{+} \cup {\mathcal L} _{-} $.
  Our aim is to construct all admissible sandwiched lattices  $L$ that 
  are invariant under 
  ${\mathcal U} _{m}  .$
  By \cite[Theorem 1.3 (A2)]{Tiep} these lattices $L$ are universally 
  strongly perfect as will be explained in Section \ref{usp} below. 
  To describe the lattices we need to restrict the spin representation
  of the orthogonal group $O_{4m}^+(2)$ to its subgroup $\GU_{2m}(\FF _4)$ 
  which is the topic of the next  paragraph.

\subsection{The spin representations of the orthogonal group.} 

The results of this section might be well known, but  we did not find them
explicitly 
in the literature. We follow the exposition of the textbook 
\cite{FultonHarris}, in particular \cite[Chapter 20]{FultonHarris},
  and  thank Jan Frahm for helpful hints. 
To avoid extra complications we restrict to the relevant case 
and only consider the algebraic group $G:=O_{4m}^+$. 
This is the automorphism group of a split quadratic space
$Q$ of dimension $4m$.
The Clifford algebra $C(Q)$ is the split central simple algebra of 
dimension $2^{4m} $ and 
$G$ acts on $C(Q)$ as algebra automorphisms preserving the 
even subalgebra 
$C_0(Q) $. 
 This action gives rise to a (projective) representation of 
 $G$ on the simple $C(Q)$-module $V$ of dimension $2^{2m} $ 
 which is in fact a linear representation  of the spin group 
$\Spin _{4m}$ and decomposes 
as the direct sum of 
 two non-isomorphic absolutely irreducible representations
 $$ V = V_+ \oplus V_-$$ called the even and odd spin representations 
 of $G$ each of dimension $2^{2m-1}$ 
 (see \cite[Proposition 20.15]{FultonHarris}).

  \cite[Proposition 20.15]{FultonHarris} analyses the modules $V_+$ and $V_-$ 
  and computes the weights occurring in these modules. 
  This allows to find the decomposition of the restrictions of the
  spin representations to the  general linear 
  unitary group $U_{2m}\leq \SO_{4m}^+$.
To state the result let 
$\chi $ be the linear character of a suitable covering group of 
$U_{2m}$ defined by $\chi(g):= (\det (g) )^{1/2} $ and 
$$\Delta  = \Delta _+ + \Delta _- : \Spin _{4m} \to \GL(V) $$ denote the spin representations of $\SO_{4m}^+$. 

\begin{theorem}
The restriction of 
$\chi \otimes \Delta $ is a linear representation of $U_{2m} $ with 
$$\chi \otimes \Delta \cong \bigoplus _{k=0}^{2m} \Lambda ^k (W ) $$
where $W$ denotes the natural $U_{2m}$-module.
In this decomposition 
$$\chi \otimes \Delta _+ \cong \bigoplus _{k=0}^{m} \Lambda ^{2k} (W ) 
\mbox{ and } 
\chi \otimes \Delta _- \cong \bigoplus _{k=1}^{m} \Lambda ^{2k-1} (W ) .
$$
\end{theorem} 

\begin{proof}
The weight lattice of the Lie algebra $\so_{4m}$ is the dual lattice
$D_{2m}^{\# }$ of the even sublattice of the standard lattice. 
So the weights are of the form 
$$(k_1,\ldots , k_{2m}) \in \ZZ ^{2m} \cup (\frac{1}{2} +\ZZ )^{2m} .$$
The proof of  \cite[Proposition 20.15]{FultonHarris}  exhibits explicit
weight vectors of the spin representation $\Delta $ for all 
$2^{2m}$ weights $(\pm \frac{1}{2} ,\ldots , \pm \frac{1}{2} ) $. 
A maximal torus in the subgroup $U_{2m}$ of $\SO_{4m}^+$ has the same rank,
so all these weights are distinct when restricted to the subalgebra.
The weight of $\chi $ is $(\frac{1}{2},\ldots , \frac{1}{2} ) $
and so the weights occurring in the restriction of 
$\chi \otimes \Delta $ to $U_{2m}$ are exactly the orbits under 
the symmetric group $S_{2m}$  of 
$$w_k:= (\underbrace{1,\ldots , 1}_k, \underbrace{0,\ldots , 0}_{2m-k}) \mbox{ for }
k=0,1,\ldots , 2m $$
where the $w_k$ for even $k$ occur in $\chi \otimes \Delta _+$ and 
those for odd $k$ in $\chi \otimes \Delta _-$. 
As $w_k$ is the highest weight of the representation $\Lambda ^k(W)$ 
the result follows.
\end{proof}

We now apply this result that is true for algebraic groups
to our special situation by restricting the representations to 
the finite groups of Lie type 
$O_{4m}^{+}(\FF _2)  \geq U _{2m}(\FF_4) $. 
In abuse of notation we denote by $V_+$ and $V_-$ the restriction
of the even and odd spin representations to $O_{4m}^+(\FF _2)$.
These are linear representations of this finite group. 
Also $\det ^{-1/2} = \det  : U _{2m}(\FF_4) \to \FF _4 ^* $ 
is a well defined linear representation. 
We put $W \cong \FF_4^{2m}$ the natural $U _{2m}(\FF_4)$ module.

  \begin{corollary} \label{decomp} 
	  The restriction of $V_+$ (resp. $V_-$) to the 
	  general unitary group is 
	  isomorphic to 
	  $$(V_+)_{|U_{2m} (\FF _4) } \cong \bigoplus _{k=0}^m \det  \otimes \Lambda ^{2k} (W)  \mbox{ resp. } 
	  (V_-)_{|U_{2m} (\FF _4) } \cong \bigoplus _{k=1}^m \det \otimes 
\Lambda ^{2k-1}( W)  $$
  \end{corollary} 

To simplify notation we denote by 
$$W_k := \det \otimes \Lambda ^k(W) .$$

\begin{remark}\label{Uk}
	The semi-linear unitary group 
	$\GU_{2m} (\FF _4)  = U_{2m} (\FF _4) : 2$ is the extension 
	of the full unitary group $U_{2m}(\FF _4)$ by the 
	Galois group of $\FF_4$ over $\FF _2$. The latter 
	interchanges the two modules $W_k$ and $W_{2m-k}$ and fixes $W_m$. 
For $0\leq k \leq m-1$ the $\FF _2 \GU _{2m } (\FF _4) $ modules
$$Y_k \mbox{ with } (Y_k ) _{|U_{2m}(\FF _4) }= W_k \oplus W_{2m-k } 
\mbox{ and } Y_m \mbox{ with } (Y_m )_{|U_{2m}(\FF _4) } = W_m$$
are  self-dual, absolutely irreducible, $\FF _2 \GU _{2m } (\FF _4) $-modules
of dimension 
$$d_k :=\dim (Y_k)  = 2 {{2m}\choose{k}} (0\leq k \leq m-1)
\mbox{ and }  d_m := \dim (Y_m) = {{2m}\choose{m}} .$$ 
\end{remark}

\subsection{The action of $\sigma $ on $W_k$}

The element $\sigma $ from Section \ref{not} is an element of 
$\GL_m(\FF _4) \leq \Aff ({\mathcal V}_{\FF _4}) $. 
The natural $U_{2m} (\FF _4)$-module then can be realized as 
$\omega $-eigenspace of $\eta $ on the natural $O_{4m}(\FF _2)$-module
and $\GL({\mathcal V}_{\FF _4} )$ is the stabilizer in 
$U_{2m}(\FF _4)$ of a maximal isotropic subspace. 
More precisely 
 we have the embedding 
$$\GL({\mathcal V}_{\FF _4}) \to U_{2m}(\FF _4) , g\mapsto \diag (g, (g^{[2]} )^{-1} ) $$
where $g^{[2]}$ is the matrix obtained by applying the Frobenius 
automorphism $x\mapsto x^2$ to all entries of $g$. 
So by Remark \ref{eta} the eigenvalues of $\sigma $ on 
the natural $U_{2m}(\FF _4)$-module $W$ are 
$$\zeta ,\zeta ^4 ,\ldots , \zeta ^{4^{m-1}} , 
\zeta ^{-2} , \zeta ^{-8} ,\ldots , \zeta ^{-2^{2m-1}}  $$
and the determinant of $\sigma $ on $W$ is $\omega \omega ^{-2} = \omega ^{-1} $
as  $\omega = \zeta \zeta^4 \cdots  \zeta ^{4^{m-1}} = \zeta ^{(4^m-1)/3} .$ 

\begin{lemma} \label{eigen}
For $0\leq k \leq 2m$ 
the  eigenvalues of $\sigma \in U_{2m}(\FF _4)$ on  $W_k$ are  the elements of 
$$\{ \omega ^{-1} \zeta ^{\sum _{i\in I}  (-2) ^{i} }  \mid I\subset \{ 0,\ldots 2m-1 \} , 
|I| = k \} .$$
\end{lemma}

\begin{proof}
Fix a basis $(e_j : j \in \{ 0,\ldots , 2m-1 \} )$  of eigenvectors of $\sigma $ of the extension  to $\FF _{4^m} $ 
of $W$ so that $\sigma (e_j ) = \zeta ^{(-2)^j} e_j$. 
Then the exterior products 
$$\{ 
e_{i_1} \wedge \ldots \wedge e_{i_k} \mid 1\leq i_1 < \ldots < i_k \leq 2m \} $$ form an eigenvector basis of $W_k$ 
where the eigenvalue of $\sigma $ on 
$ e_{i_1} \wedge \ldots \wedge e_{i_k}$ 
 is  $ \omega ^{-1} \zeta ^{\sum _{j=1}^k (-2)^{i_j} } $.
\end{proof}

To distinguish between the two spin representations we compare 
$2$-weights of the exponents of the eigenvalues of $\sigma $ 
as defined in Notation \ref{wt2}.

\begin{lemma}\label{EO}
For $I\subseteq \{ 1,\ldots , 2m \} $ with $|I| = k$ let 
$0\leq u<  2^{2m}-1 $ be such that 
 $$\zeta ^u = \omega ^{-1} \zeta ^{\sum _{i\in I}  (-2) ^{i} } .$$
 Then $O(u) - E(u) = m-k $. 
 In particular the $\wt_2(u)$ is 
even if and only if $m-k$ is even. 
\end{lemma}
 
\begin{proof}
	We have 
	$$\omega ^{-1} \zeta ^{\sum _{i\in I}  (-2) ^{i} } = 
	\zeta ^b \mbox{ with } b = \sum_{i=0}^{2m-1} b_i 2^i  \mbox{ and } 
	b_i \in \{ 0 , -1 \} $$ 
	such that $b_i = -1 $ if and only if either $i\in I$ is odd 
	or $i \not\in I$ and $i$ is even. 
	As $\zeta ^{2^{2m}-1}  = 1$ and $2^{2m}-1 = \sum _{i=0}^{2m-1}2^i$ 
	we may multiply $\zeta ^b$ by $\zeta ^{2^{2m}-1}  = 1$ to 
	obtain $\zeta ^b = \zeta ^a$ with 
  $a=\sum_{i=0}^{2m-1} a_i 2^i $ such that  $a_i = 1+b_i \in \{0,1 \}$.
  Then $E(a) = | \{ i \in I \mid i\mbox{ even } \} | $ and 
  $O(a) = |\{ i \in \{0,\ldots , 2m-1\} \setminus I \mid i \mbox{ odd } \} |$.
  In particular $O(a)-E(a)  $ equals the number of odd numbers in $\{ 0,\ldots , 2m-1 \}$ minus the cardinality of $I$, so $O(a) - E(a) = m-k $.
\end{proof}

\begin{corollary}\label{eigenvalsU}
	The eigenvalues of $\sigma $ on 
	$Y_k$ are exactly  the elements of $\Theta _k$ 
	from Notation \ref{wt2}. 
	We have 
 $1\in \Theta _k$ if and only if $k=m$, and then
	the eigenvalue $1$ of $\sigma $ occurs twice in $Y_m$.
\end{corollary}

Comparing the eigenvalues of $\sigma $ on $V_+$ and $V_-$ with
the ones obtained in Proposition \ref{eigenquot} we find 

\begin{corollary} \label{evenoddoddeven}
	If $m$ is even 
	then $\BW_{2m}/2\BW^{\# } _{2m} \cong V_+$ 
	and $\BW^{\# } _{2m} /\BW_{2m}  \cong V_-$.
\\
	If $m$ is odd
	then $\BW^{\# } _{2m} /\BW_{2m}  \cong V_+$ 
	and $\BW_{2m}/2\BW^{\# } _{2m} \cong V_-$.
\end{corollary}

\section{The ${\mathcal U}_m$ invariant sandwiched lattices} 

\subsection{The ${\mathcal U}_m$ invariant sandwiched lattices}

The results of the previous  section (in particular Corollary \ref{decomp} 
in combination with Remark \ref{Uk}) 
can be summarized to find all lattices 
$\Lambda \in {\mathcal L}_{-} $  and $\Gamma \in {\mathcal L}_{+}$
invariant under ${\mathcal U}_m = 2_+^{1+4m}.\GU_{2m}(\FF _4) $ 
where ${\mathcal L}_{-}$ and $ {\mathcal L}_{+}$ are as in 
Definition \ref{sigmasandwich}.
Note that the lattices $\Gamma $ are even lattices whereas 
only $\sqrt{2} \Lambda $ is even. 
Recall 
from Remark \ref{Uk} that $d_k$ denotes the dimension of 
the absolutely irreducible ${\mathcal U}_m$-module $Y_k$.

\begin{theorem} \label{latuni}
(a) 
$$ \BW _{2m} / 2\BW^{\#} _{2m} \cong \bigoplus _{k\in M_+} Y_{k} $$ 
as an $\FF _2 \GU_{2m} (\FF _4)$ module. 
The ${\mathcal U}_m $ invariant lattices $\Gamma \in {\mathcal L}_{+} $ 
are in bijection with the subsets $J\subseteq M_+$, 
such that $\Gamma _J /2\BW _{2m}^{\#}  \cong \bigoplus _{k \in J } Y_{k }$ 
and satisfy $2\Gamma _{J}^{\#} = \Gamma _{M_+ \setminus J}$. 
The discriminant group is 
$$\Gamma _J^{\#} / \Gamma _J \cong 
(\ZZ/2\ZZ )^{2^{2m-1}} \oplus (\ZZ/4\ZZ)^{\sum_{k\in M_+\setminus J} d_{k } } .$$
(b) 
$$ \BW _{2m}^{\#} / \BW _{2m} \cong \bigoplus _{k\in M_-} Y_{k} $$ 
as an $\FF _2 \GU_{2m} (\FF _4)$ module. 
The ${\mathcal U}_m $ invariant lattices $\Lambda \in {\mathcal L}_{-} $
are in bijection with the subsets $I\subseteq M_-$, 
such that $\Lambda _I /\BW _{2m} \cong \bigoplus _{k \in I } Y_{k}$ 
and satisfy $\Lambda _{I}^{\#} = \Lambda _{M_-\setminus I}$.
$\sqrt{2}\Lambda _I$ is an even lattice with discriminant group  
$$(\sqrt{2} \Lambda _I)^{\#}/(\sqrt{2} \Lambda _I)  \cong 
(\ZZ/2\ZZ )^{2^{2m-1}} \oplus (\ZZ/4\ZZ)^{\sum_{k\in M_-\setminus I} d_{k } } .$$
\end{theorem}

\begin{figure}
  \centering
    \def\svgscale{0.5}
      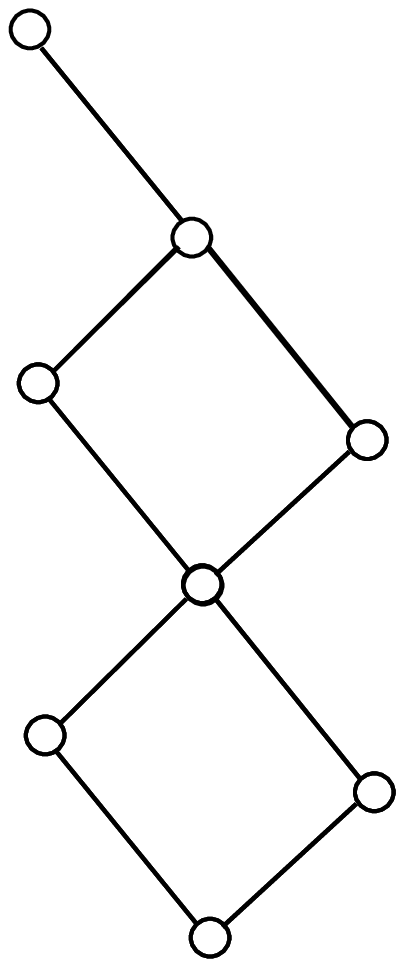
      \caption{Duality}
\end{figure}

\begin{proof}
	The module structure of the quotients of the two lattices
	follows from Corollaries \ref{decomp} and \ref{evenoddoddeven}. 
	To simplify notation we place ourselves into situation (a). 
	The ${\mathcal U}_m$ invariant lattices 
$\Gamma $ with                
$2\BW^{\#} _{2m} \subseteq \Gamma \subseteq \BW_{2m} $
are in bijection with the $\GU_{2m}(\FF _4) $ invariant 
submodules of 
$ \BW _{2m} / 2\BW^{\#} _{2m} = \bigoplus _{k\in M_+} Y_{k} $. 
As all the $Y_{k}$ are pairwise non-isomorphic simple 
$\FF_2 \GU_{2m}(\FF _4) $-modules, the invariant submodules 
correspond to subsets of  $M_+$.
As all the $Y_{k}$ are self-dual, so 
$$2\Gamma ^{\#} / 2\BW^{\#} _{2m} \cong \BW _{2m} / \Gamma $$
from which one gets the duality as illustrated in Figure 1.
Moreover $2\Gamma _J^{\#} \cap \Gamma _J = 2\BW^{\#} _{2m} $ and 
$2\Gamma _J^{\#} + \Gamma _J = \BW _{2m }$ 
implies that 
$$2 (\Gamma _J^{\#} / \Gamma _J) = \BW_{2m} /\Gamma _J 
\cong \bigoplus _{k\in M_+\setminus J } Y_{k}  .$$ 
Together with  
$$|\Gamma _J^{\#} / \Gamma _J | = |\BW_{2m}^{\#} /\BW_{2m} | 
\cdot |\BW_{2m} /\Gamma _J  | \cdot | \Gamma _J^{\#} /\BW_{2m}^{\# } | $$
we obtain the structure of the discriminant group. 
\\
Part (b) is proved with the same arguments. 
\end{proof}

\subsection{The automorphism group of the lattices $\Gamma _J$ and $\Lambda _I$} 
\begin{theorem}
	For all $\emptyset \neq J \subset M_+ $ 
	we have 
	$\Aut(\Gamma _J) = {\mathcal U}_m $. 
	\\
	For all $\emptyset \neq I \subset M_- $ 
	we have 
	$\Aut(\Lambda _I) = {\mathcal U}_m $. 
\end{theorem}

\begin{proof}
Let $J$ be a proper subset of $M_+$.
Then $\Gamma _J + 2 \Gamma _J^{\#} = \BW_{2m} $, so by construction
$$ {\mathcal U}_m \leq  \Aut(\Gamma _J) \leq \Aut (\BW _{2m} ) = {\mathcal G}_{2m}.$$
Moreover  
$\Aut (\Gamma _J) \neq {\mathcal G}_{2m} $ because 
$\BW_{2m}/2\BW_{2m}^{\#} $ is a simple ${\mathcal G}_{2m}$-module. 
As $\Gamma U_{2m} (\FF _4) $ is a maximal subgroup 
of $O_{4m}^+(2)$ (see for instance \cite[Theorem 3.12]{Wilson})  
also ${\mathcal U}_m$ is a maximal subgroup of ${\mathcal G}_{2m} $
so ${\mathcal U}_m =  \Aut(\Gamma _J) $.
The statement for $\Lambda _I$ is proved similarly as 
$\Lambda _I \cap \Lambda _I^{\#} = \BW_{2m} $.
\end{proof}

\subsection{Construction $\Dcyc $ for
the lattices $\Gamma _J$ and $\Lambda _I$} \label{last}

In this section we show that 
the lattices $\Gamma _J$ and $\Lambda _I$ from Theorem \ref{latuni}
coincide with the lattices ${\mathcal L}(\widehat{(\cC _{\star J} )}) $
and ${\mathcal L}(\widehat{(\cC _{\star I} )}) $
from Section \ref{unitarysand}. 

\begin{theorem}\label{GammaequalsL}
\begin{itemize}
\item[(a)] 
For $\emptyset \neq J\subset M_+ $ the lattice 
$\Gamma_J$ from Theorem \ref{latuni} is given by
$$\Gamma _J  = \begin{cases} 
	2 {\mathcal L}(\widehat{(\cC _ {\star J})} ) & m\not\in J  \\
	{\mathcal L}(\widehat{(\cC _ {\star J})} ) & m\in J.
\end{cases} .$$ 
In particular if 
$\{m,m-2\} \cap J = \emptyset $, then 
$\min (\Gamma _J) = 2^{m+1} = \min (2\BW_{2m}^{\#} ) $. 
\item[(b)] 
For $\emptyset \neq I\subset M_- $ the lattice 
$\Lambda _I$ from Theorem \ref{latuni} is given by
$$\Lambda _I  = {\mathcal L}(\widehat{(\cC _{\star I})} ) .$$ 
In particular if 
$m-1 \not\in  I  $, then 
$\min (\Lambda _I) = \min (\BW _{2m} )  = 2^{m} $. 
\end{itemize}
\end{theorem}

\begin{proof}
The lattices $\Lambda _I $ are clearly $\sigma $ invariant,
and hence in ${\mathcal L}_{-} $. 
Moreover by Corollary \ref{eigenvalsU}  all 
$\Gamma _J$ are admissible and hence  
in ${\mathcal L}_{+} $.
So we may use Proposition \ref{labellat} to identify the lattices. 
By Corollary \ref{eigenvalsU} the eigenvalues 
of $\sigma $ on 
$\Lambda _I/\BW _{2m} $ (respectively $\Gamma _J/2\BW_{2m}^{\#} $)
are exactly the elements of 
$\bigcup _{k\in I} \Theta _k $
respectively $\bigcup _{k\in J} \Theta _k $.
These coincide with the eigenvalues of $\sigma $ on 
${\mathcal L}(\widehat{(\cC _{\star I})} ) / \BW_{2m}  $,
${\mathcal L}(\widehat{(\cC _{\star J})} )/2\BW_{2m}^{\#} $ (if $m\in J$), 
respectively
$2{\mathcal L}(\widehat{(\cC _{\star J})} )/2\BW_{2m}^{\#} $ (if $m\not\in J$)
as given in Theorem \ref{sandlatdetails}.
\end{proof}

\begin{corollary}\label{SBW}
	Let $m\geq 3$.  
	\begin{itemize}
		\item[(a)]
	For $J_0:=M_+ \setminus \{ m,m-2 \} $ 
	the lattice $\Gamma _{J_0}$ has minimum $2^{m+1}$ 
	and discriminant group 
	$$\Gamma _{J_0}^{\#}/\Gamma _{J_0} \cong (\ZZ/2\ZZ )^{2^{2m-1}} \oplus 
	(\ZZ/4\ZZ) ^{{2m \choose m} + 2{2m \choose m-2}}.$$
	If $m=3$ then $J_0 = \emptyset $ so 
	$\Gamma _{J_0} = 2\BW_{2m}^{\#} $. 
		\item[(b)]
For $I_0:=M_-\setminus \{ m -1 \} $, the rescaled lattice
$\SBW_{2m} := \sqrt{2} \Lambda _{I_0}$ is an even lattice of 
minimum $2^{m+1}$ and 
discriminant group 
$$(\SBW_{2m})^{\#}/(\SBW_{2m}) \cong 
	(\ZZ/2\ZZ )^{2^{2m-1}} \oplus 
	(\ZZ/4\ZZ) ^{2{2m \choose m-1}}.$$
	\end{itemize}
\end{corollary}

	For $m\geq 3$ 
	the lattice $\SBW_{2m}$ 
	has the maximum density among the
	unitary invariant sandwiched lattices that we considered in this 
	paper. In particular these lattices are denser than the 
	Barnes-Wall lattices in the same dimension. 
	More precisely we compute the 2-adic logarithm 
	of the center density (as defined in \cite[Chapter 1, Formula (27)]{SPLAG}) 
	of $\SBW_{2m} $ as 
	$$ \log_2(\delta (\SBW_{2m})) = 
	(2m-3)2^{2m-2} - 2{2m \choose m-1} $$ 
	which we tabulate for the first few values of $m$ 
	$$\begin{array}{|r|rrrrrrrr|} 
		\hline 
			m & 3 & 4 &  5 & 6 & 7 & 8 & 9 & 10 \\
		\hline 
			\log_2(\delta (\SBW_{2m})) & 18 & 208 &
			1372 &
			7632 &
			39050 &
			190112 &
			895524 &
			4120528  \\
		\hline 
		\end{array} .$$
Though these lattices are denser than the 
Barnes-Wall lattices of the same dimension, 
they do  not improve on the asymptotic density of the 
Barnes-Wall lattices as given in \cite[Chapter 1, Formula (30)]{SPLAG}.

\section{Strongly perfect lattices} \label{usp}

The notion of strongly perfect lattices has been introduced by 
Boris Venkov (see \cite{Venkov} for a comprehensive introduction). 

\begin{definition}
	A lattice $L$ is {\em strongly perfect}, if 
	its minimal vectors form a spherical 4-design.
\end{definition}

One interest of strongly perfect lattices stems from the fact that 
they provide examples of locally densest lattices. 
Another point comes from the connection to Riemannian geometry: 
Recall that a lattice $L$ is called {\em universally strongly perfect},
if all non-empty layers $L_a:=\{ \ell \in L \mid (\ell, \ell ) = a \} $ 
form spherical 4-designs. 
It has been shown in \cite{Coulangeon} that 
 universally perfect lattices  achieve local minima of
Epstein's zeta function.

One method to show that a lattice is universally strongly perfect has
been used by Bachoc in \cite{Bachoc}, where she shows that all
layers of the Barnes-Wall lattices form spherical 6-designs. 

It is based on the following proposition, used in several places of the
relevant literature. 
\begin{proposition} (see e.g. \cite[Proposition 2.5]{LempkenTiep})  \label{invari}
	Let $G\leq O_n(\RR) $ be a finite subgroup of the compact 
	real orthogonal group. Assume that all $G$ invariant homogeneous
	polynomials of degree $\leq 4$ are also invariant under $O_n(\RR) $.
	Then all $G$-orbits in $\RR^n$ form spherical 4-designs. 
\end{proposition}

\begin{theorem}\label{BW'usp}
	All the lattices $\Gamma _{J}$ 
	and $\Lambda _I$ from Theorem \ref{latuni} 
	are universally strongly perfect.
\end{theorem}

\begin{proof}
	We show that the assumption of Proposition \ref{invari} holds 
	for ${\mathcal U}_m = 2^{1+4m}_+ . \GU_{2m}(\FF _4) \leq O_{2^{2m}}(\RR )$. 
Then the theorem follows as all layers of such invariant 
lattices are disjoint unions of 
${\mathcal U}_m$-orbits. 
To compute the invariant harmonic polynomials we use the fact that 
${\mathcal U}_m= {\mathcal C}_m(4^H_{\bf 1} )$ (see \cite[Proposition 7.3.1]{NebeRainsSloane}).
Therefore by \cite[Corollary 5.7.5]{NebeRainsSloane} the space of 
homogeneous invariants of ${\mathcal U}_m$ of degree $d$ is spanned 
by the genus $m$ complete weight enumerators of Hermitian 
self-dual codes $C= C^{\perp }\leq \FF_4^d$ of length $d$ 
containing the all ones vector. 
By the classification of these codes, there are up to 
coordinate permutation unique such codes
of lengths 2 and 4, the repetition code $i_2 = \langle (1,1) \rangle \leq \FF_4^2 $ 
and its orthogonal sum $i_2 \perp i_2 \leq \FF_4^4$.
The genus $m$ complete weight enumerator of $i_2$ is 
the $O_{2^{2m}}(\RR ) $ invariant quadratic form $q$ and 
the one of $i_2\perp i_2$ is $q^2$. 
So all invariants of ${\mathcal U}_m$ of degree $2$ and $4$ are also invariant under $O_{2^{2m}}(\RR )$.
As all layers of any ${\mathcal U}_m$ invariant lattice are disjoint unions of ${\mathcal U}_m$-orbits 
we conclude that all these layers form spherical 4-designs. So
all ${\mathcal U}_m$ invariant lattices are universally strongly perfect. 
\end{proof}

Note that this theorem also follows from
\cite[Theorem 1.3 (A2)]{Tiep}.

\section{Examples in small dimension}

\begin{figure}[!h]
  \centering
   \begin{minipage}{0.45\textwidth}
  \vspace*{0.8cm}
   \centering
     \def\svgscale{0.4}
    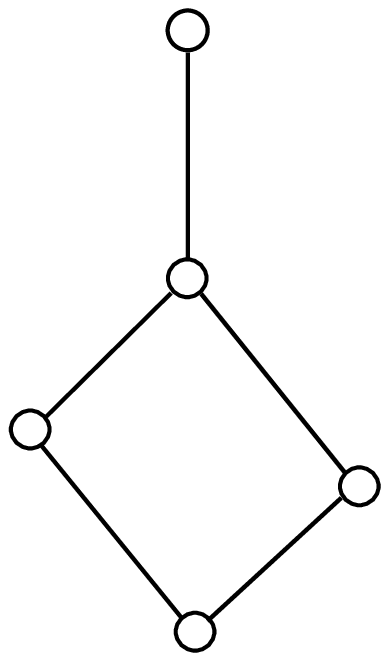
    \caption{$m=2$}
   \end{minipage}
     \begin{minipage}{0.45\textwidth}
  \centering
   \def\svgscale{0.4}
        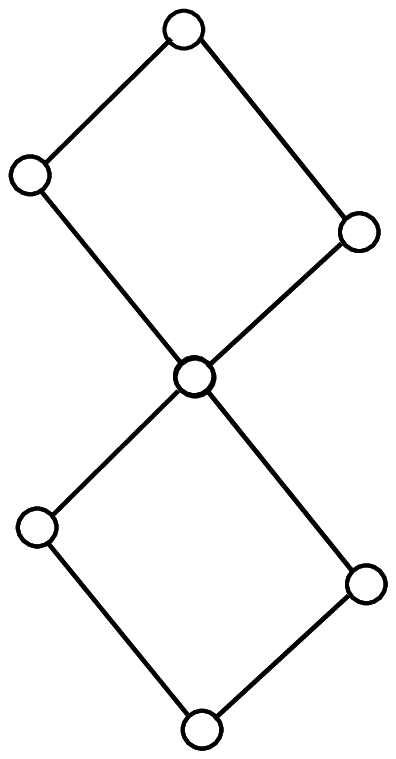
     \caption{$m=3$}
\end{minipage}
\end{figure}

In dimension 16 (so $m=2$) 
we find two new universally strongly perfect lattices:
$\Gamma _{\{ 2 \} }$ and its dual $\Gamma _{\{ 2 \}}^{\#} = \frac{1}{2} \Gamma _{\{ 0\} } $.
The discriminant groups are 
$$\Gamma _{\{ 2 \}}^{\#}/\Gamma _{\{ 2 \} } \cong 
\ZZ/2\ZZ ^8 \oplus \ZZ/4\ZZ ^2 \mbox{ and } 
\Gamma _{\{ 0 \}}^{\#}/\Gamma _{\{ 0 \} } \cong 
\ZZ/2\ZZ ^8 \oplus \ZZ/4\ZZ ^6 .$$ For the minimum we compute
$$\min (\Gamma _{\{ 2 \}} ) = \min (\BW _4) = 4, 
\min (\Gamma _{\{ 0 \} } ) = 6 $$  
so the Hermite function $\gamma $ with $\gamma (L) = \frac{\min(L)}{\det(L)^{1/\dim(L)}} $ rounded to 2 decimal places are 
	$$\gamma (\BW_4 ) \sim 2.83 , \ \gamma (\Gamma_{\{ 2 \} }) \sim 2.38  , \ 
	\gamma (\Gamma _{\{ 0 \} }) \sim 2.52 .$$
The kissing numbers are computed with Magma as
$$|\Min (\BW_4 ) | = 4320 , \ 
|\Min (\Gamma _{\{ 2 \}} ) | = 864 , \ 
|\Min (\Gamma _{\{ 0 \}} ) | = 1536  .$$

For dimension 64 (so $m=3$) we list the invariants of the lattices 
as computed with Magma in the following table:
$$\begin{array}{|r|r|r|r|r|}
	\hline 
	\mbox{ name } & \mbox{ smith } & \min & \mbox{ kissing } & \mbox{ Hermite }\\
	\hline
	\BW _6 & 1^{32} 2^{32} & 8 & 9,694,080  & 5.66  \\ 
	\Gamma _{\{3\}} &  1^{20} 2^{32} 4^{12} & 8 & 114,048 & 4.36 \\ 
	\Gamma _{\{1\}}  & 1^{12} 2^{32} 4^{20}  & 12 & 4,257,792 &  5.50\\
	\frac{1}{\sqrt{2}} \SBW_6 =\Lambda _{\{0\}} &  \frac{1}{2}^{2} 1^{32} 2^{30} & 8 &  9,694,080 & 5.91 \\
	\Lambda _{\{2 \}} &  \frac{1}{2}^{30} 1^{32} 2^2 & 4 & 2,395,008  & 5.42\\
	\hline 
\end{array} 
$$


\begin{thebibliography}{11}
\bibitem{Abdukhalikov}
K.  Abdukhalikov, 
Defining sets of extended cyclic codes invariant under the affine group. J. Pure Appl. Algebra 196 (2005) 1--19.
\bibitem{Bachoc} C. Bachoc, Designs, groups and lattices. 
 J. Th\'eor. Nombres Bordeaux 17 (2005) 25--44. 
\bibitem{BarnesSloane} E.S. Barnes, N.J.A. Sloane, 
	New lattice packings of spheres. Can. J. Math. 35 (1983) 117--130
\bibitem{BarnesWall} E.S. Barnes, G.E. Wall, Some extreme forms defined in 
	terms of abelian groups. J. Austral. Math. Soc. 1 (1959/1961)  47--63
\bibitem{BroueEnguehard} M. Brou\'e, M. Enguehard, 
		Une famille infinie de formes quadratiques enti\`eres; 
		leurs groupes d'automorphismes. 
		Ann. Sci. l'ENS {6} (1973) 17--51 
	\bibitem{SPLAG} J.H.Conway, N.J.A.Sloane, {\em Sphere packings, lattices and groups.} 
	Grundlehren der Mathematischen Wissenschaften {290},
		Springer-Verlag, New York, 1988.
	\bibitem{Coulangeon} 
		R.  Coulangeon,  Spherical designs and zeta functions of lattices. Int. Math. Res. Not. 2006, Art. ID 49620, 16 pp.
	\bibitem{FultonHarris} 
		W. Fulton, J.  Harris, {\em Representation theory. A first course.} Graduate Texts in Mathematics, 129. Readings in Mathematics. Springer-Verlag, New York, 1991. 
	\bibitem{KL} P. Kanwar, S.R.L\'opez-Permouth, 
		Cyclic Codes over the Integers Modulo $p^m$. 
		Finite fields and their applications 3 (1997) 334--352. 
	\bibitem{Oggier} 
		W. Kositwattanarerk, F.  Oggier,
		Connections between Construction D and related constructions of lattices. 
		Des. Codes Cryptogr. 73 (2014) 441--455.
	\bibitem{LempkenTiep} 
		W. Lempken, B. Schr\"oder, P.H. Tiep, 
		Symmetric squares, spherical designs, and lattice minima.
		With an appendix by Christine Bachoc and Tiep.
		 J. Algebra  240  (2001) 185--208.
	 \bibitem{MacWilliamsSloane} J. MacWilliams, N.J.A. Sloane, 
		 {\em The Theory of Error Correcting Codes.} North Holland (1977)
	 \bibitem{norm}
		 G. Nebe,
		 The normaliser action and strongly modular lattices.
		 Enseign. Math. 43 (1997) 67--76. 
	 \bibitem{CliffordNRS} 
		 G. Nebe, E.M. Rains, N.J.A. Sloane, 
		 The invariants of the Clifford groups. Des. Codes Cryptogr. 24 (2001) 99--121.
	 \bibitem{NebeRainsSloane} 
		 G. Nebe, E.M. Rains, N.J.A. Sloane, 
		 {\em Self-dual codes and invariant theory.}
		 Algorithms and Computation in Mathematics, 17. Springer-Verlag, Berlin, 2006.
	 \bibitem{Quebbemann} 
		 H.-G. Quebbemann,
		 Modular lattices in Euclidean spaces. J. Number Theory 54 (1995) 190--202.
	 \bibitem{Serre} J. P. Serre, {\em Linear Representations of Finite Groups.} Springer Graduate Texts in Mathematics 42, 
		 Springer-Verlag, New York-Heidelberg,  1977
	\bibitem{Tiep}
P.H. Tiep, Finite groups admitting Grassmannian 4-designs.
 J. Algebra  306  (2006) 227--243.
 \bibitem{Venkov} B. Venkov, {\it R\'eseaux et designs sph\'eriques.}
	 Monogr. Ens. Math. {37} (2001) 10-86.
 \bibitem{Wall} G.E. Wall, On the Clifford collineation, transform and similariy groups (IV) An application to quadratic forms. 
	 Nagoya Math. J. {21} (1962) 199--222.
 \bibitem{Wilson} 
	 R.A. Wilson, {\em The finite simple groups.} Graduate Texts in Mathematics, 251. Springer-Verlag London, 2009. 
\end{thebibliography}
\end{document}